\documentclass[12pt]{article}

\usepackage{graphicx}
\usepackage{latexsym,amssymb}
\usepackage{amsthm}
\usepackage{indentfirst}
\usepackage{amsmath}

\newtheorem{Theorem}{Theorem}[section]
\newtheorem*{Theorem A}{Theorem A}

\newtheorem{Definition}[Theorem]{Definition}

\newtheorem{Lemma}[Theorem]{Lemma}
\newtheorem{Sublemma}[Theorem]{Sublemma}

\newtheorem*{Remark}{Remark}
\newtheorem{Corollary}[Theorem]{Corollary}

\newtheorem*{Claim}{Claim}
\newtheorem*{Acknowledgements}{Acknowledgements}

 \def\NN{{\mathbb N}} 

 \def\RR{{\mathbb R}}

\def\sv{{\rm sv}}

   \def\cN{{\cal N}} 
    
   \def\cP{{\cal P}}

\def\dim{\operatorname{dim}}

\def\Sing{\operatorname{Sing}}

\def\supp{\operatorname{supp}}

\begin{document}

\title{On the singular hyperbolicity of star flows}

\author{Yi Shi \and Shaobo Gan \and Lan Wen}
\date{}

\maketitle

\begin{abstract}
We prove for a generic star vector field $X$ that if, for every
chain recurrent class $C$ of $X$, all singularities in $C$ have the
same index, then the chain recurrent set of $X$ is singular
hyperbolic. We also prove that every Lyapunov stable chain recurrent
class of a generic star vector field is singular hyperbolic. As a
corollary, we prove that the chain recurrent set of a generic
4-dimensional star flow is singular hyperbolic.
\end{abstract}

\section{Introduction}

Let $M^d$ be a $d$-dimensional $C^\infty$ compact Riemannian
manifold without boundary. Denote by ${\cal X}^1(M^d)$ the space of
$C^1$ vector fields on $M^d$, endowed with the $C^1$ topology. A
vector field $X\in{\cal X}^1(M^d)$ generates a $C^1$ flow
$\phi_t=\phi^X_t$ on $M^d$, as well as the tangent flow $\Phi_t={\rm
d}\phi_t$ on $TM^d$. Denote by ${\rm Sing}(X)$ the set of
singularities of $X$, and ${\rm Per}(X)$ the set of periodic points
of $X$. A singularity or a periodic orbit of $X$ are both called a
{\em critical orbit} or a {\em critical element} of $X$.

A compact invariant set $\Lambda$ of $X$ is {\em hyperbolic} if there are two constants
 $C\ge 1, \lambda>0$, and a continuous $\Phi_t$-invariant splitting
$$
T_\Lambda M^d= E^s\oplus \langle X\rangle\oplus E^u
$$
such that for every $x\in \Lambda$ and $t\ge 0$,
\begin{eqnarray*}
\|\Phi_t|_{E^s(x)}\|&\le& C{\rm e}^{-\lambda t},\\
\|\Phi_{-t}|_{E^u(x)}\|&\le& C{\rm e}^{-\lambda t}.
\end{eqnarray*}
Here  $\langle X(x)\rangle$ denotes the space spanned by  $X(x)$, which  is
0-dimensional if $x$ is a singularity, or 1-dimensional if $x$ is
regular. If $\Lambda$ consists of a critical element, denote  the
{\it index} of $\Lambda$ by ${\rm Ind}(\Lambda)={\rm dim} E^s$.

Let $\phi_t$ be the flow generated by a vector field $X$. For any
$\varepsilon>0, T>0$, a finite sequence $\{x_i\}_{i=0}^n$ on $M$ is
called \emph{$(\varepsilon,T)$-chain} of $X$ if there are $t_i\ge T$
such that $d(\phi_{t_i}(x_i),x_{i+1})<\varepsilon$ for any $0\le
i\le n-1$. For $x,y\in M^d$, one says that $y$ is \emph{chain
attainable} from $x$ if there exists $T>0$ such that for any
$\varepsilon>0$, there is an $(\varepsilon,T)$-chain
$\{x_i\}_{i=0}^n$ with $x_0=x$ and $x_n=y$. If $x$ is chain
attainable from itself, then $x$ is called a \emph{chain recurrent
point}. The set of chain recurrent points is called \emph{chain
recurrent set} of $X$, denoted by ${\rm CR}(X)$.

Chain attainability is a closed equivalence relation on ${\rm
CR}(X)$. For each $x\in{\rm CR}(X)$, the equivalence class $C(x)$
(which is compact) containing $x$ is called the chain recurrent
class of $x$. A chain recurrent class is called {\it trivial} if it
consists of a single critical element. Otherwise it is called {\it
nontrivial}. Since every hyperbolic critical element $c$ of $X$ has
a well-defined continuation $c_Y$ for $Y$ close to $X$, the chain
recurrent class $C(c)$ also has a well-defined continuation $C(c_Y,
Y)$.

A compact invariant set $\Lambda$ is called {\it chain transitive}
if for every pair of points $x,y\in \Lambda$,  $y$ is chain
attainable from $x$, where all chains are chosen in $\Lambda$. Thus
a chain recurrent class is just a maximal chain transitive set, and
every chain transitive set is contained in a unique chain recurrent
class.

A vector field $X\in {\cal X}^1(M^d)$ is called a {\em star vector
field} or a {\em star flow}, if it satisfies the {\em star
condition}, i.e.,  there exists a $C^1$ neighborhood $\cal U$ of $X$
such that every critical element of every $Y\in \cal U$ is
hyperbolic. The set of $C^1$ star vector field on $M^d$ is denoted
by ${\cal X}^*(M^d)$.

The notion of star system came up from the study of the famous
stability conjecture. Recall that a classical theorem of Smale
\cite{Sma70} (for diffeomorphisms) and Pugh-Shub \cite{PS70} (for
flows) states that Axiom A plus the no-cycle condition implies the
$\Omega$-stability. Palis and Smale \cite{PaSm70} conjectured that
the converse also holds, which has been known as the
$\Omega$-stability conjecture. In the study of the conjecture,
Pliss, Liao and Ma\~n\'e noticed an important condition called  (by
Liao) the {\it star condition}. As defined above, the star condition
looks quite weak because, though involving perturbations, it
concerns critical elements only, and the hyperbolicity considered is
in an individual but not uniform way. Indeed, the $\Omega$-stability
implies the star condition easily (Franks \cite{Fra71} and Liao
\cite{Lia79}). Thus whether the star condition could give back Axiom
A plus the no-cycle condition became a striking problem, raised by
 Liao \cite{Lia81} and Ma\~n\'e \cite{Man82}. An affirmative answer to the problem
would, of course, contain the $\Omega$-stability conjecture. For
diffeomorphisms, Aoki \cite{Aok92} and Hayashi \cite{Hay92} proved
that the star condition indeed implies Axiom A plus the no-cycle
condition. For flows, there are counterexamples if the flow has a
singularity. For instance, the geometric Lorenz attractor
\cite{Guc76}, which has a singularity, is a star flow but fails to
satisfy Axiom A. In fact, Liao \cite{Lia81} and Ma\~n\'e
\cite{Man82} raised this problem for {\it nonsingular} star flows,
and hence it was known as the {\it nonsingular star flow problem}.
The problem was solved by Gan-Wen \cite{GaW06} proving that
nonsingular star flows do satisfy Axiom A and the no-cycle
condition.

These give rise to a  new problem
--- to understand {\it singular star flows}, of which the geometric Lorenz attractor is one of the
typical models. Note that, while being not structurally stable, the
Lorenz attractor is quite robust under perturbations. Analytically,
 while being not hyperbolic, it exhibits quite some contractions and expansions. How to
describe such a dynamics? Morales, Pacifico and Pujals \cite{MPP04}
have given an appropriate notion about it, called singular
hyperbolicity, which is of central importance to the subject. Their
definition is for dimension 3, and the following higher dimensional
version can be found in \cite{ZGW08, MeM08}.

\begin{Definition} {\em (Positive singular hyperbolicity)}
Let $\Lambda$ be a compact invariant set of $X\in {\cal X}^1(M^d)$.
We say that $\Lambda$ is {\em positively singular hyperbolic} of $X$
if there are constants $C\ge 1$ and $\lambda>0$, and a continuous
invariant splitting
$$
T_\Lambda M=E^{ss}\oplus E^{cu}
$$
 w.r.t. $\Phi_t$ such that, for all $x\in \Lambda$ and $t\geq0$, the
 following three conditions are satisfied:

 $(1)$  $E^{ss}$ is $(C,\lambda)-$dominated by $E^{cu}$, i.e., $\|\Phi_t|_{E^{ss}(x)}\|\cdot\|\Phi_{-t}|_{E^{cu}(\phi_t(x)}\|\leq C{\rm e}^{-\lambda t}$.

 $(2)$  $E^{ss}$ is uniformly contracting, i.e., $\|\Phi_t|_{E^{ss}(x)}\|\leq C{\rm e}^{-\lambda t}$.

 $(3)$  $E^{cu}$ is {\em sectionally expanding}, i.e., for any 2 dimensional subspace $L\subset E^{cu}(x)$,
        $$
        |\det \left(\Phi_t|_L\right)|\ge C^{-1}{\rm e}^{\lambda t}.
        $$

\end{Definition}

 We say that
$\Lambda$ is {\em negatively singular hyperbolic} of $X$ if
$\Lambda$ is positively singular hyperbolic of $-X$.

 A union of
finitely many positively singular hyperbolic sets is positively
singular
  hyperbolic. Likewise for the negative case.

\begin{Definition} {\em (Singular hyperbolicity)}
We say that $\Lambda$ is {\em singular hyperbolic} of $X$ if it is
either positively singular hyperbolic of $X$, or negatively singular
hyperbolic of $X$, or a disjoint union of a positively singular
hyperbolic set of $X$ and a negatively singular hyperbolic set of
$X$.
\end{Definition}

Using the notion of singular hyperbolicity, the following conjecture
was formulated in \cite{ZGW08}:

\vskip 5mm \noindent{\bf Conjecture.\ }\cite{ZGW08} {\em For every
star vector field $X\in {\cal X}^*(M^d)$, the chain recurrent set
${\rm CR}(X)$ is singular hyperbolic and consists of finitely many
chain recurrent classes.} \vskip 5mm

\begin{Remark}
The conjecture is open even in 2-dimensional case.
\end{Remark}

In this paper we obtain some partial results to this conjecture. Let us
 say that a set $C$ {\it has a homogeneous index for
singularities} if all the singularities in $C$ have the same index.
Here are the main theorems of this paper.

\vskip 5mm \noindent{\bf Theorem A.\ } {\em There is a dense
$G_\delta$ set ${\mathcal G}_A\subset{\cal X}^*(M^d)$ such that, for
every $X\in{\mathcal G}_A$, if a chain recurrent class $C$ of $X$
has a homogeneous index for singularities, then $C$ is positively or negatively singular
hyperbolic.} \vskip 5mm

\begin{Remark}
The homogeneity requirement here looks restrictive.
However, we will prove that, for generic star vector fields, any
chain recurrent class can have at most two different indices for its
singularities.
\end{Remark}

A direct consequence  is the following

\vskip 5mm \noindent{\bf Theorem B.\ } {\em There is a dense
$G_\delta$ set ${\mathcal G}_B\subset{\cal X}^*(M^d)$ such that, for
every $X\in{\mathcal G}_B$, if every chain recurrent class $C$ of
$X$ has a homogeneous index for singularities, then the chain
recurrent set ${\rm CR}(X)$ is singular hyperbolic.} \vskip 5mm

The next theorem states that,  for generic star vector fields, if a
chain recurrent class is Lyapunov stable, then it is singular
hyperbolic.

\vskip 5mm \noindent{\bf Theorem C.\ } \footnote{\noindent Theorem C is claimed in \cite{AMS12} under the assumption of the homogeneous property, i.e., the conclusion of our Theorem \ref{Thm:Loc-Homoge}.}{\em There is a dense
$G_\delta$ set ${\mathcal G_C}\subset{\cal X}^*(M^d)$ such that, for
every $X\in{\mathcal G_C}$,  every Lyapunov stable chain recurrent
class of $X$ is positively singular hyperbolic.} \vskip 5mm

These theorems allow us to achieve the singular hyperbolicity of
chain recurrent set in the 4 dimensional case.

\vskip 5mm \noindent{\bf Theorem D.\ } {\em There is a dense
$G_\delta$ set ${\mathcal G_D}\subset{\cal X}^*(M^4)$ such that, for
every $X\in{\mathcal G_D}$, the chain recurrent set ${\rm CR}(X)$ is
singular hyperbolic.} \vskip 5mm

We also obtain a  description of ergodic measures of star flows,
which could be thought as the counterpart of hyperbolic measures for
diffeomorphisms. The following theorem is derived from a powerful
shadowing lemma of Liao \cite{Lia85} and the estimation of size of
invariant manifolds of Liao \cite{Lia89}.

\vskip 5mm \noindent{\bf Theorem E.\ } {\em If $\mu$ is an ergodic
measure of a star flow, then $\mu$ is a hyperbolic measure.} \vskip
5mm

Theorem A, C and D are proved  in Section 3 by admitting two
technical theorems that will be proved in Section 4 and 5
respectively. A detailed version of Theorem E will be proved in
Section 5 too.

\begin{Acknowledgements}
We are very grateful for the invaluable suggestions of the anonymous referee.
This work is partially supported by the Balzan research Project of J. Palis.
YS is supported by Chinese Scholarship Council.
SG is supported by 973 project 2011CB808002, NSFC 11025101 and 11231001.
LW is supported by NSFC 11231001.
\end{Acknowledgements}

\section{Preliminaries}

\subsection{Flows associated to a vector field}

Given $X\in{\cal X}^1(M^d)$, $X$ generates a $C^1$ flow $\phi_t:M^d\to M^d$, and the tangent flow
$\Phi_t={\rm d}\phi_t:TM^d\to TM^d$.

%We denote ${\cal P}:~TM^d\to M^d$ the bundle projection.

The usual linear Poincar\'e flow $\psi_t$ is defined as following. Denote the normal bundle of $X$ by
$$
{\cal N}={\cal N}^X=\bigcup_{x\in M^d\setminus\Sing (X)} {\cal N}_x,
$$
where ${\cal N}_x$ is the orthogonal complement of the flow direction $X(x)$, i.e.,
$$
{\cal N}_x=\{v\in T_xM^d: v\perp X(x)\}.
$$
Given $v\in {\cal N}_x$, $x\in M^d\setminus\Sing(X)$, $\psi_t(v)$ is the orthogonal projection of $\Phi_t(v)$ on
${\cal N}_{\phi_t(x)}$ along the flow direction, i.e.,
$$
\psi_t(v)=\Phi_t(v)-\frac{\langle \Phi_t(v), X(\phi_t(x))\rangle}{\|X(\phi_t(x))\|^2}X(\phi_t(x)),
$$
where $\langle\cdot, \cdot\rangle$ is the inner product on $T_xM$ given by the Riemannian metric.

We will need another flow $\psi_t^*:{\cal N}\to{\cal N}$, which is called {\em scaled linear Poincar\'e flow}. Given $v\in {\cal N}_x$, $x\in M^d\setminus\Sing(X)$,
$$
\psi_t^*(v)=\frac{\|X(x)\|}{\|X(\phi_t(x))\|}\psi_t(v)=\frac{\psi_t(v)}{\|\Phi_t|_{\langle X(x)\rangle}\|},
$$
where $\langle X(x)\rangle$ is the 1-dimensional subspace of $T_xM^d$ spanned by the vector $X(x)\in T_xM^d$. In a shadowing lemma of Liao (see Theorem \ref{Thm:Liao-Shadowing}), it is required some hyperbolicity with respect to this scaled linear Poincar\'e flow on the orbit arc.

The next lemma states the basic properties of star flows, proved in
\cite{Lia79}.

\begin{Lemma}\label{Lem:Basic-Property}(\cite{Lia79})
For any $X\in{\cal X}^*(M^d)$, there is a $C^1$ neighborhood $\cal U $ and numbers $ \eta>0$ and $T>0$ such that for any periodic orbit $\gamma$ of $Y\in\cal U$ with period $\pi(\gamma)\ge T$, if ${\cal N}_{\gamma}=N^s\oplus N^u$ is the hyperbolic splitting with respect to $\psi^Y_t$ then
\begin{itemize}
\item For every $x\in\gamma$ and $t\ge T$, one has
$$
\frac{\|\psi^Y_t|_{N^s(x)}\|}{m(\psi^Y_t|_{N^u(x)})}\le {\rm e}^{-2\eta t};
$$

\item For every $x\in\gamma$, then
$$
\prod_{i=0}^{[\pi(\gamma)/T]-1}\|\psi^Y_T|_{N^s(\phi^Y_{iT}(x))}\|\le {\rm e}^{-\eta\pi(\gamma)},
$$
$$
\prod_{i=0}^{[\pi(\gamma)/T]-1}m(\psi^Y_T|_{N^u(\phi^Y_{iT}(x))})\ge {\rm e}^{\eta\pi(\gamma)}.
$$
Here $m(A)$ is the mini-norm of $A$, i.e., $m(A)=\|A^{-1}\|^{-1}$.
\end{itemize}
\end{Lemma}
Let $E$ be a finitely dimensional vector space. Denote $\wedge^2 E$ the second exterior power of $E$.
Given a linear isomorphism: $A:E\to F$ between finitely dimensional vector spaces $E$ and $F$, denote $\wedge^2 A: \wedge^2 E\to\wedge^2 F$ the linear isomorphism induced by $A$. Now the second item of last theorem has the following consequence:

\begin{Corollary}\label{Cor:2-Form-Basic-Property}
For any $X\in{\cal X}^*(M^d)$, there is a $C^1$ neighborhood $\cal U $ and numbers $ \eta>0$ and $T>0$ such that for any periodic orbit  $\gamma$ of $Y\in\cal U$ with period $\pi(\gamma)\ge T$, if ${\cal N}_{\gamma}=N^s\oplus N^u$ is the hyperbolic splitting with respect to $\psi^Y_t$,  $E^{cs}=N^s\oplus\langle X\rangle$ and $E^{cu}=N^u\oplus\langle X\rangle$ which are invariant subbundles of $\Phi^Y_t$, then we have for any $x\in\gamma$,
$$
\prod_{i=0}^{[\pi(\gamma)/T]-1}\|\wedge^2\Phi^Y_T|_{E^{cs}(\phi^Y_{iT}(x))}\|\le {\rm e}^{-\eta\pi(\gamma)},
$$
$$
\prod_{i=0}^{[\pi(\gamma)/T]-1}m(\wedge^2\Phi^Y_T|_{E^{cu}(\phi^Y_{iT}(x))})\ge {\rm e}^{\eta\pi(\gamma)}.
$$
\end{Corollary}

\begin{Remark}
For simplicity, we will assume the constant $T=1$.
\end{Remark}

\subsection{$C^1$ connecting and generic results for flows}
We need the following two versions of connecting lemmas.
\begin{Lemma}\label{Lem:WX-Connecting}(\cite{WeX00})
For any vector field $X\in{\cal X}^1(M^d)$ and any neighborhood ${\cal U}$ of $X$, for any point $z\notin{\rm Per}(X)\cup{\rm Sing}(X)$, there exist $L>0, \rho>1, \delta_0>0$ such that for any $\delta\in (0, \delta_0]$,  for any $p$ and $q$ in $M\setminus \Delta$ ($\Delta=\cup_{0\le t\le L}\phi_t^X(B_\delta(z)$), if both the positive orbit of $p$ and the negative orbit of $q$ enter into $B_{\delta/\rho}(z)$, then there is $Y\in{\cal U}$ such that
\begin{itemize}
\item $q$ is on the positive orbit of $p$ with respect to the flow $\phi_t^Y$ generated by $Y$.

\item $Y(x)=X(x)$ for any $x\in M\setminus \Delta$.
\end{itemize}
\end{Lemma}

The connecting lemma of chains is also true for all the star flows, since all the critical elements of star flows are hyperbolic (see \cite{BoC04}).
\begin{Lemma}\label{Lem:Pseudo-Connecting}(\cite{BoC04})
Let $X\in{\cal X}^*(M^d)$. For any $C^1$ neighborhood $\cal U$ of $X$ and $x,y\in M^d$, if y is chain attainable from x, then there exists $Y\in \cal U$ and $t>0$ such that $\phi^Y_t(x)=y$. Moreover, for every $k\ge 1$, let $\{x_{i,k}, t_{i,k}\}_{i=0}^{n_k}$ be a $(1/k, T)$-chain from $x$ to $y$ and denote by
$$
\Lambda_k=\bigcup_{i=0}^{n_k-1}\phi_{[0,t_{i,k}]}(x_{i,k}).
$$
Let $\Lambda$ be the upper Hausdorff limit of $\Lambda_k$, i.e., $\Lambda$ consists of points $z$ such that there exist $z_k\in \Lambda_k$ and $\lim_{k\to\infty}z_k=z$. Then for any neighborhood $U$ of $\Lambda$, there exists $Y\in {\cal U}$ with $Y=X$ on $M\setminus U$ and $t>0$ such that $\phi^Y_t(x)=y$.
\end{Lemma}

\begin{Remark}
According to the proof of the above connecting lemma for chain (\cite{BoC04}), the conclusion can be strengthened as following:  for any neighborhood $U$ of $\Lambda$, and for any finitely many (hyperbolic) critical elements $c_i, i=1,2,\cdots, j$, there exist a neighborhood $V_i$ of $c_i (i=1,2,\cdots, j)$ and $Y\in {\cal U}$ with $Y=X$ on $(\cup_{i=1}^j V_j)\cup(M\setminus U)$ and $t>0$ such that $\phi^Y_t(x)=y$.
This strong version will be used in the proof of Lemma \ref{Lem:Expon-Sing}.
\end{Remark}

We need the following generic properties for star vector fields.

\begin{Lemma}\label{Lem:Generic-Property}
There is a dense $G_\delta$ set ${\mathcal G}\subset{\cal X}^*(M^d)$ such that for any $X\in{\mathcal G}$, one has
\begin{enumerate}
\item For every critical element $p$ of $X$, the chain recurrent class $C(p)=C(p_X,X)$ is continuous at $X$ in the Hausdorff topology.
\item If $p$ and $q$ are two different critical elements of $X$ with $C(p)=C(q)$, then there exists a $C^1$ neighborhood $\cal U$ of $X$ such that for any $Y\in{\cal U}$, one has $C(p_Y,Y)=C(q_Y,Y)$.
\item For any hyperbolic critical element $p$ of $X$, if $W^u(p)\subset C(p)$, then there is a $C^1$ neighborhood $\cal U$ of $X$ such that for any $Y\in{\cal U}$, $C(p_Y,Y)$ is Lyapunov stable.
\item For any nontrivial chain recurrent class $C$ of $X$, there exists a sequence of periodic orbits $Q_n$ such that $Q_n$ tends to $C$ in the Hausdorff topology.
\end{enumerate}
\end{Lemma}

\begin{Remark}
Item 1, 2 and 3 is from \cite{GaY11} and item 4 is from \cite{Cro06}.
\end{Remark}
\section{Reducing the main theorems to two technical results}

In this section we reduce the proofs of the main theorems to two
technical theorems, Theorem 3.4 and 3.5. First we define the saddle
value of a singularity, a crucial value for the analysis of
singularities whose chain recurrent class is nontrivial.

\begin{Definition}\label{Def:Saddle-Value}
Let $X\in{\cal X}^1(M^d)$ and $\sigma$ a hyperbolic singularity of $X$. Assume the Lyapunov exponents of $\Phi_t(\sigma)$ are
$$
\lambda_1\leq\cdots\leq\lambda_s<0<\lambda_{s+1}\leq\cdots\leq\lambda_d,
$$
then the saddle value ${\rm sv}(\sigma)$ of $\sigma$ is defined as
$$
{\rm sv}(\sigma)=\lambda_s+\lambda_{s+1}.
$$
\end{Definition}

\begin{Definition}\label{Def:Lorenz-like}
Let $X\in{\cal X}^1(M^d)$ and $\sigma$ a hyperbolic singularity of $X$. Assume that $C(\sigma)$ is nontrivial and the Lyapunov exponents of $\Phi_t(\sigma)$ are
$$
\lambda_1\leq\cdots\leq\lambda_s<0<\lambda_{s+1}\leq\cdots\leq\lambda_d.
$$
We say $\sigma$ is {\em Lorenz-like}, if the following conditions are satisfied:
\begin{itemize}
    \item $\sv(\sigma)\neq0$.
    \item If $\sv(\sigma)>0$, then  $\lambda_{s-1}<\lambda_s$, and $W^{ss}(\sigma)\cap C(\sigma)=\{\sigma\}$. Here $W^{ss}(\sigma)$ is the invariant manifold corresponding to the bundle $E^{ss}_\sigma$ of the partially hyperbolic splitting $T_\sigma M=E^{ss}_\sigma\oplus E^{cu}_\sigma$, where $E^{ss}_\sigma$ is the invariant space corresponding to the Lyapunov exponents $\lambda_1, \lambda_2, \cdots, \lambda_{s-1}$ and $E^{cu}_\sigma$ corresponding to the Lyapunov exponents $\lambda_s, \lambda_{s+1}, \cdots, \lambda_{d}$.

    \item If $\sv(\sigma)<0$, then  $\lambda_{s+1}<\lambda_{s+2}$, and $W^{uu}(\sigma)\cap C(\sigma)=\{\sigma\}$. Here $W^{uu}(\sigma)$ is the invariant manifold corresponding to the bundle $E^{uu}_\sigma$ of the partially hyperbolic splitting $T_\sigma M=E^{cs}_\sigma\oplus E^{uu}_\sigma$, where $E^{cs}_\sigma$ is the invariant space corresponding to the Lyapunov exponents $\lambda_1, \lambda_2, \cdots, \lambda_{s+1}$ and $E^{uu}_\sigma$ corresponding to  the Lyapunov exponents $\lambda_{s+2}, \lambda_{s+3}, \cdots, \lambda_{d}$.
\end{itemize}
\end{Definition}

\begin{Remark}
If the singularity $\sigma$ is Lorenz-like, then the splitting (say, $T_\sigma M=E^{ss}_\sigma\oplus E^{cu}_\sigma$ in the case ${\rm sv}(\sigma)>0$) is a singular hyperbolic splitting over $\{\sigma\}$.
\end{Remark}

Although in the definition of Lorenz-like singularity (and singular hyperbolicity) it is allowed that $E^{uu}_\sigma$ is trivial (for ${\rm sv}(\sigma)<0$), i.e., $E^{uu}_\sigma=\{0\}$, we will show that for $C^1$ generic star vector field $X$, if $C(\sigma)$ is nontrivial, then $E_\sigma^{uu}$ should be nontrivial (see Theorem \ref{Thm:degenerate} below). We need the important Main Theorem of Liao in \cite{Lia89} (see \cite{YaZ13} for a generalization):

\begin{Theorem}(\cite[Main Theorem]{Lia89})\label{Thm:finite}
Given $X\in{\cal X}^*(M)$, there exists a neighborhood $\cal U$ of $X$ such that
$$
\sup_{Y\in{\cal U}}\#\{P\subset M: P \mbox{ is a periodic sink of } Y\}<\infty.
$$
\end{Theorem}

\begin{Theorem}\label{Thm:degenerate}
There exists a dense $G_\delta$ subset ${\cal G}_0\subset {\cal X}^*(M)$ such that for any $X\in{\cal G}_0$ and any singularity $\sigma$ of $X$, if $T_\sigma M$ is sectional contracting or sectional expanding, then $C(\sigma)$ is trivial.
\end{Theorem}

\begin{proof}
We only consider sectional contracting singularities.
Define a map
$$
N: {\cal X}^*(M)\to\mathbb{N}
$$
by
$$
N(X)=\#\{P\subset M: P \mbox{ is a periodic sink of } X\}.
$$
According to Theorem \ref{Thm:finite}, $N(X)$ is well-defined. Since $N(\cdot)$ is lower semi-continuous, there exists a dense $G_\delta$ subset ${\cal G}_0\subset {\cal X}^*(M)$ such that $N(\cdot)$ is continuous on ${\cal G}_0$. Given $X\in{\cal G}_0$, take a small neighborhood ${\cal U}\subset {\cal X}^*(M)$ of $X$ such that $N(\cdot)$ is constant on ${\cal U}$.

We will prove that for any singularity $\sigma$ of $X\in{\cal G}_0$, if $T_\sigma M$ is sectional contracting, then $C(\sigma)$ is trivial. Otherwise, assume that $C(\sigma)$ is nontrivial. Then according to $C^1$ connecting lemma (Lemma \ref{Lem:Pseudo-Connecting}), there exists $Y\in{\cal U}$ such that $Y\equiv X$ in a neighborhood of $\sigma$, which implies that $T_\sigma M$ is still sectional contracting for $Y$, and $Y$ has a homoclinic loop $\Gamma$ associated to $\sigma=\sigma_Y$. $\Gamma\cup\{\sigma\}$ is sectional contracting since the unique invariant measure is the atomic measure $\delta_\sigma$ supported on $\sigma$. It is easy to see that there is a sequence $Y_n$ tending $Y$ and periodic orbit $P_n$ of $Y_n$ tending to $\Gamma\cup\{\sigma\}$ in the Hausdorff topology. Since the invariant measure supported on $P_n$ converges to $\delta_\sigma$, $P_n$ is a sink of $Y_n$ for $n$ large enough and hence $N(Y_n)\ge N(Y)+1$. This contradicts that $N(\cdot)$ is constant on ${\cal U}\ni Y$.
\end{proof}

From now on, we will only consider singularities which are neither sectional contracting nor sectional expanding.

\begin{Definition}\label{Def:Periodic-Index}
Let $X\in{\cal X}^*(M^d)$ and $\sigma\in {\rm Sing}(X)$ such that $C(\sigma)$ is nontrivial. Then the periodic index ${\rm Ind}_p(\sigma)$ of $\sigma$ is defined as
$$
{\rm Ind}_p(\sigma)=\left\{\begin{array}{ll} s, & {\rm if\ } \sv(\sigma)<0,\\
s-1, & {\rm if\ } \sv(\sigma)>0.\end{array}\right.
$$
For a periodic orbit $P$ of $X$, we define ${\rm Ind}_p(P)={\rm Ind}(P)$.
\end{Definition}

\begin{Remark} The notion of
periodic index of singularity is to describe the index of periodic
orbits derived from the perturbation of homoclinic loop associated
to the corresponding singularity. Our definition does not concern
the case that the saddle value of singularity is zero, which could
not occur if we admit the generic assumptions. However, we will
prove in Lemma \ref{Lem:Expon-Sing} that for every $X\in{\cal
X}^*(M^d)$ and $\sigma\in{\rm Sing}(X)$, if $C(\sigma)$ is
nontrivial, then $\sv(\sigma)\neq0$. This result justifies our
definition.
\end{Remark}

The next theorem  studies the singularities of a nontrivial  chain
recurrent class for a generic star flow. We show that these
singularities are all Lorenz-like, that is, the tangent space of the
singularity admits a partially hyperbolic splitting, and the strong
stable/unstable manifold intersects the chain recurrent class only
at the singularity. The proof will be given in Section 4.

\begin{Theorem}\label{Thm:Anal-sing}
For any $X\in{\cal X}^*(M^d)$ and $\sigma\in {\rm Sing}(X)$, if the chain recurrent class $C(\sigma)$ is nontrivial, then any singularity $\rho\in C(\sigma)$ is Lorenz-like. Moreover,
there is a dense $G_\delta$ subset ${\mathcal G_1}\subset{\cal X}^*(M^d)$ and if we further assume that $X\in{\mathcal G_1}$, then ${\rm Ind}_p (\rho)={\rm Ind}_p (\sigma)$.
\end{Theorem}

\begin{Remark} From this
theorem and the definition of periodic index of singularity, it
follows that, for a generic star vector field $X$ and any nontrivial
chain recurrent class $C(\sigma)$ of $X$, if $\rho\in
C(\sigma)\cap{\rm Sing}(X)$, then the index of $\rho$ can only be
${\rm Ind}_p (\sigma)+1$ if $\sv(\rho)>0$, or ${\rm Ind}_p
(\sigma)$ if $\sv(\rho)<0$.
\end{Remark}

The next theorem states that if the singularities of a chain
recurrent class are all Lorenz-like and have the same index, then
the chain recurrent class is singular hyperbolic. The proof will be
given in Section 5.

\begin{Theorem}\label{Thm:Sing-Hyper}
There is a dense $G_\delta$ subset ${\mathcal G_2}\subset{\cal X}^*(M^d)$ such that for any $X\in{\mathcal G_2}$ and $\sigma\in{\rm Sing}(X)$, if $C(\sigma)$ is nontrivial and for any singularity $\rho\in C(\sigma)$, ${\rm Ind} (\rho)={\rm Ind} (\sigma)$, then $C(\sigma)$ is positively or negatively singular hyperbolic.
\end{Theorem}

\begin{Remark}
Notice that Theorem \ref{Thm:Anal-sing} talks about the {\em
periodic index} of singularities, while Theorem \ref{Thm:Sing-Hyper}
talks about the {\em index} (not periodic index) of singularities.
\end{Remark}

Now we give the proofs of Theorem A, C and D by assuming Theorem
\ref{Thm:Anal-sing} and \ref{Thm:Sing-Hyper}. A detailed version of
Theorem E (Theorem \ref{Thm:Erg-Measure}) will be proved in section
5. \vskip 5mm

\noindent{\bf Proof of Theorem A.\ } Let ${\mathcal G_A}={\mathcal G_2}$, which is  a dense $G_\delta$ subset of
${\cal X}^*(M^d)$. Let $X\in{\mathcal G_A}$, and $C$ be a chain
recurrent class  of $X$. If $C\cap{\rm Sing}(X)=\varnothing$, then
we apply \cite{GaW06} to conclude that $C$ is a hyperbolic set,
which is of course singular hyperbolic. Now, assume that there exists some
singularity $\sigma\in C$. If $C=\{\sigma\}$, from star condition, $C$ is hyperbolic and hence singular hyperbolic. If $C$ is nontrivial, Theorem \ref{Thm:Sing-Hyper} tells us that $C$ is positively or negatively singular hyperbolic. This proves Theorem A.\hfill $\Box$

\vskip 5mm

\noindent{\bf Proof of Theorem C.\ } We let ${\mathcal
G_C}={\mathcal G_0}\cap{\mathcal G_1}\cap{\mathcal G_2}$. Consider any $X\in{\mathcal
G_C}$ and any Lyapunov stable chain recurrent class $C$ of $X$. If
$C\cap{\rm Sing}(X)=\varnothing$, then \cite{GaW06} guarantees that
$C$ is a hyperbolic attractor. So we only need to consider the case
when $C$ contains  some singularity. Since $C$ is Lyapunov
stable, we must have $W^u(\sigma)\subset C$ for any $\sigma\in C\cap{\rm
Sing}(X)$.

\begin{Claim}
For any $\sigma\in C\cap{\rm Sing}(X)$, we have that $\sv(\sigma)>0$.
\end{Claim}
\begin{proof}[Proof of Claim: ]
Otherwise, assume that there exists $\sigma\in C\cap{\rm Sing}(X)$, $\sv(\sigma)<0$.
By Theorem \ref{Thm:Anal-sing}, $\sigma$ is Lorenz-like, i.e., there exists a negatively singular hyperbolic splitting $T_\sigma M=E^{cs}_\sigma\oplus E^{uu}_\sigma$. According to Theorem \ref{Thm:degenerate}, $T_\sigma M$ is not sectional contracting. So, $E^{uu}_\sigma$ is nontrivial. Hence, $W^{uu}(\sigma)\setminus C\not=\emptyset$, which contradicts $W^{u}(\sigma)\subset C$.
\end{proof}
Now the singularities in $C$ have the same index. Applying
Theorem \ref{Thm:Sing-Hyper} we conclude that $C$ is positively singular
hyperbolic. This proves Theorem C.\hfill $\Box$ \vskip 5mm

Combining these results, we could show that the
singular hyperbolicity of chain recurrent set for generic star flows
in dimension 4. \vskip 5mm

 \noindent{\bf Proof of Theorem D.\ } We
assume ${\rm dim}(M)=4$ and ${\mathcal G_D}={\mathcal G_0}\cap{\mathcal
G_1}\cap{\mathcal G_2}\cap{\mathcal G}$ which is a dense $G_\delta$
subset of ${\cal X}^*(M^4)$, where ${\cal G}$ is the dense
$G_\delta$ set in Lemma \ref{Lem:Generic-Property}. As in the proofs
of the above theorems, for any $X\in{\mathcal G_D}$, we only need to
consider a nontrivial chain recurrent class $C$ of $X$ such that
there exists $\sigma\in C\cap{\rm Sing}(X)$.

If there exists some singularity $\rho\in C$ such that ${\rm
Ind}(\rho)=3$, then ${\rm dim}(E^u(\rho))=1$ and $W^u(\rho)$ has two
separatrices. Since we assume $X\in{\mathcal G}$, $C(\rho_X,X)=C$
depends continuously on $X$ and hence is robustly nontrivial.

\begin{Claim}
$W^u(\rho)\subset C$ and, consequently, $C$ is Lyapunov stable.
\end{Claim}
\begin{proof}[Proof of Claim: ]
In fact, suppose on the contrary that one separatrix ${\rm
Orb}(x_1)$ of $W^u(\rho)$ is not contained in $C$. By the upper
semi-continuity of chain recurrent class, we know this holds
robustly. The non-triviality of $C(\rho)$ implies the other
separatrix ${\rm Orb}(x_2)$ of $W^u(\rho)$ is contained in $C$.
Using the connecting lemma for chains, you can perturb ${\rm
Orb}(x_2)$ to be the homoclinic orbit associated to $\rho$. Then
applying the $\lambda$-lemma, an arbitrarily small perturbation
could make the positive orbit of $x_2$ arbitrarily close to $x_1$,
which is no longer contained in $C(\rho)$. Combining all these
perturbations together, we get a vector field $Y$ arbitrarily
$C^1$ close to $X$, such that
$$
W^u(\rho_Y)\cap C(\rho_Y,Y)={\{\rho_Y}\},
$$
contradicting the fact that $C(\rho_X,X)$ is robustly nontrivial.
\end{proof}

From the claim and Theorem C, $C$ is positively singular hyperbolic.

If there are some singularity $\rho\in C$ such that ${\rm Ind}(\rho)=1$, we just need to consider $-X$. Then following the analysis above directly, $C$ is Lyapunov stable for $-X$, which is negatively singular hyperbolic for $X$. So we can reduce to the case that all the singularities contained in $C$ have the same index 2, which allows us to applying Theorem A. As a result, $C$ is singular hyperbolic.

Now we have proved that every chain recurrent class of $X$ is
singular hyperbolic. And hence,  ${\rm CR}(X)$ is
singular hyperbolic. This proves Theorem D. \hfill $\Box$ \vskip 5mm

\section{Analysis of singularities}

In this section, we will analyze the singularities contained in a nontrivial chain recurrent class for some $X\in{\cal X}^*(M^d)$. Our main technique is the extended linear Poincar\'e flow introduced in $\cite{LGW05}$, which has been proved to be a useful tool in the analysis of non-isolated singularities (e.g., see \cite{ZGW08, GaY13, AMS12}).

First we state a lemma on the estimation of index of periodic orbits
which accumulate on singularities and their homoclinic orbits. Then
we use the dominated splitting of the extended linear Poincar\'e
flow to achieve the properties of Lyapunov exponents of
singularities. Especially, we will conclude that all the
singularities whose chain recurrent class are nontrivial  are
Lorenz-like.
\begin{Lemma}\label{Lem:Index-Estimation}
Let $X\in{\cal X}^*(M^d)$,  $\sigma\in{\rm Sing}(X)$ and
$\Gamma={\rm Orb}(x)$ be a homoclinic orbit associated to $\sigma$.
Assume that there exists a sequence of star vector fields $\{X_n\}$
converging to $X$ in the $C^1$ topology and periodic orbit $P_n$ of
$X_n$ with index $l$ such that $\{P_n\}$ converges to $\Gamma\cup\{\sigma\}$ in the
Hausdorff topology. Then there exist two subspaces $E, F\subset T_\sigma M$ such that
\begin{enumerate}
\item  $E$ is $(l+1)$-dimensional and sectional contracting:
$$
\frac{1}{k}\sum_{i=0}^{k-1}\log\|\wedge^2\Phi_1^X|_{\Phi_i^X(E)}\|\leq-\eta,\ k=1,2,\cdots
$$
\item  $F$ is $(d-l)$-dimensional and sectional expanding:
$$
\frac{1}{k}\sum_{i=0}^{k-1}\log m(\wedge^2\Phi_1^X|_{\Phi_i^X(F)})\geq\eta,\ k=1,2,\cdots
$$
\end{enumerate}
Here the constant $\eta$ comes from corollary \ref{Cor:2-Form-Basic-Property}.

Moreover, we have the following estimation of the index of periodic orbits:
$$
{\rm Ind}(\sigma)-1\leq l={\rm Ind}(P_n)\leq{\rm Ind}(\sigma)~.
$$

\end{Lemma}

\begin{proof}
Let the hyperbolic splitting of $P_n$ be
$$
T_{P_n}M=E^s(P_n)\oplus  \langle X_n(P_n)\rangle\oplus E^u(P_n).
$$

Consider the $X_n$-invariant subspace
$$E_n=E^s(P_n)\oplus\langle X_n(P_n)\rangle$$
on $P_n$. Since $P_n$ tends to the homoclinic loop associated to $\sigma$, their periods must tend to infinity as $n\to\infty$. For $n$ large enough, you can apply Corollary \ref{Cor:2-Form-Basic-Property} to get the following estimations
$$
\prod_{i=0}^{[\pi(x_n)]-1}\|\wedge^2\Phi^{X_n}_1|_{E_n(\phi^{X_n}_{i}(x_n))}\|\leq
{\rm e}^{-\eta\pi(x_n)}
$$
for any $x_n\in P_n={\rm Orb}(x_n)$. Then for any $\epsilon>0$, Pliss Lemma (\cite{Pli72}) gives some point $p_n\in P_n$ satisfying
$$
\frac 1k \sum_{i=0}^{k-1}\log \|\wedge^2\Phi^{X_n}_1|_{\Phi^{X_n}_i(E_n(p_n))}\|\leq -\eta+\epsilon,\ k=1,2,\cdots
$$

Assume $p_n$ tends to $y\in\Gamma\cup\{\sigma\}$. Taking some subsequence if necessary, one can assume $E_n(p_n)\to E(y)$, then we have
$$
\frac 1k \sum_{i=0}^{k-1}\log \|\wedge^2\Phi^{X}_1|_{\Phi^{X}_i(E(y))}\|\leq -\eta+\epsilon,\ k=1,2,\cdots
$$
Now the Pliss Lemma $\cite{Pli72}$ allows us to find $n_j\to\infty$ such that
$$
\frac 1k \sum_{i=0}^{k-1}\log \|\wedge^2\Phi^{X}_1|_{\Phi^{X}_{i+n_j}(E(y))}\|\leq -\eta+2\epsilon,\ k=1,2,\cdots
$$
Since $\phi_{n_j}(y)$ tends to $\sigma$ as $j\to\infty$, we derive a subspace $E\subset T_{\sigma}M$ with ${\rm dim}E={\rm dim}E_n(p_n)=l+1$ and
$$
\frac 1k \sum_{i=0}^{k-1}\log \|\wedge^2\Phi^{X}_1|_{\Phi^{X}_{i}(E)}\|\le -\eta+2\epsilon,\ k=1,2,\cdots
$$
So, $E$ is sectional contracting under $\Phi^X_t$. Notice that we can choose the constant $\epsilon$ arbitrarily small, this give us the proof of first item.

For the second item, we only need to consider $-X$.

Now for the estimation of the index of $P_n$, if we assume ${\rm Ind}(\sigma)<l={\rm Ind}(P_n)$, then
$$
{\rm dim}(E\cap E^u(\sigma))\geq{\rm dim}E+{\rm dim}E^u(\sigma)-d\geq l+1+d-(l-1)-d=2~.
$$

However, since $E$ is sectional contracting and $E^u(\sigma)$
is sectional expanding, this is absurd. So  $l={\rm Ind}(P_n)\leq{\rm Ind}(\sigma)$.
For the other side of the inequality, we
only need to consider $-X$, and the same argument as above will show
that $l={\rm Ind}(P_n)\geq{\rm Ind}(\sigma)-1$. This finishes the
proof of the lemma.
\end{proof}

\begin{Remark}
From this lemma and its proof, one can see
\begin{itemize}
\item If some periodic orbit is sufficiently close to a homoclinic loop associated to some singularity $\sigma$ of a star flow, then the index of the periodic orbit could only be ${\rm Ind}(\sigma)-1$ or ${\rm Ind}(\sigma)$.
\item In this lemma, we do not need to assume that the star flow is generic.
\end{itemize}
\end{Remark}

~

Let us recall some basic definitions in \cite{LGW05}. Denote by
$$
G^1=G^1(M^d)=\{L:L \textrm{ is a 1-dimensional subspace of } T_xM^d, x\in M^d \}
$$
the Grassmannian manifold of $M^d$. Given $X\in{\cal X}^1(M^d)$, the tangent flow $\Phi_t$ induces a flow
\begin{eqnarray*}
\Phi_{t}:G^1 & \to & G^1 \\
L & \mapsto & \Phi_{t}(L)
\end{eqnarray*}
on $G^1$.

Let $\beta :G^1\to M^d$ and $\xi :TM^d\to M^d$ be the corresponding bundle projections. It naturally induces a (pullback) bundle
$$
\beta^*(TM^d)=\{(L,v)\in G^1\times TM^d: \beta (L)=\xi (v)\}.
$$
Then $\beta^*(TM^d)$ is a $d$-dimensional vector bundle over $G^1$ with the bundle projection
$$\iota:\beta^*(TM^d)\to G^1$$
$$\iota(L,v)=L.$$

Then we could lift the tangent flow $\Phi_t$ to $\beta^*(TM^d)$, which is called {\it extended tangent flow}, (still) denoted by
$$\Phi_t:\beta^*(TM^d)\to\beta^*(TM^d)$$
$$\Phi_t(L,v)=(\Phi_t(L),\Phi_t(v)).$$
Let
$$\cP=\{(L,v)\in\beta^*(TM^d): v\in L\}.$$
This is a $1$-dimensional subbundle of $\beta^*(TM^d)$ over $G^1$, which is invariant under any extended tangent flow.
Similarly, we could define the normal bundle of $\cP$ as follows
$$\cN=\cP^\perp=\{(L,v)\in\beta^*(TM^d): v\perp L\}.$$
Then $\cN$ is a $(d-1)$-dimensional subbundle of $\beta^*(TM^d)$ over $G^1$. Now for every $X\in{\cal X}^1(M^d)$, we could define the {\it extended Poincar\'e flow} of $X$
$$
\psi_t=\psi^X_t: \cN\to\cN
$$
to be
$$
\psi_t(L,v)=\pi(\Phi_t(L,v)), \qquad\forall~ (L,v)\in\cN,
$$
where $\pi$ is the orthogonal projection from $\beta^*(TM^d)$ to $\cN$ along $\cP$.

For a compact invariant set $\Lambda$ of $X\in{\cal X}^1(M^d)$, we denote
$$
B(\Lambda)=\{L\in G^1:\beta(L)\in\Lambda, \exists X_n\to X, p_{n}\in{\rm Per}(X_n), {\rm Orb}(p_n,X_n)\hookrightarrow_n\Lambda,
$$
$$
\qquad\textrm{ such that }\langle X_n(p_n)\rangle\to L\}.
$$

$$
B^j(\Lambda)=\{L\in G^1:\beta(L)\in\Lambda, \exists X_n\to X, p_{n}\in{\rm Per}(X_n),{\rm Ind}(p_n)=j,
$$
$$
\qquad{\rm Orb}(p_n,X_n)\hookrightarrow_n\Lambda, \textrm{ such that }\langle X_n(p_n)\rangle\to L\}.
$$
Here ${\rm Orb}(p_n,X_n)\hookrightarrow_n\Lambda$ means that the Hausdorff upper limit of ${\rm Orb}(p_n,X_n)$ is contained in $\Lambda$.

~

\begin{Lemma}\label{Lem:Expon-Sing}
Let $X\in{\cal X}^*(M^d)$ and $\sigma\in {\Sing}(X)$. Assume that the Lyapunov exponents of $\Phi_t(\sigma)$ are
$$
\lambda_1\leq\cdots\leq\lambda_s<0<\lambda_{s+1}\leq\cdots\leq\lambda_d.
$$
If $C(\sigma)$ is nontrivial, then
\begin{enumerate}
\item either $\lambda_{s-1}\neq\lambda_s$ or $\lambda_{s+1}\neq\lambda_{s+2}$.
\item if $\lambda_{s-1}=\lambda_s$, then $\lambda_s+\lambda_{s+1}<0$.
\item if $\lambda_{s+1}=\lambda_{s+2}$, then $\lambda_s+\lambda_{s+1}>0$.
\item if $\lambda_{s-1}\neq\lambda_s$ and $\lambda_{s+1}\neq\lambda_{s+2}$, then $\lambda_s+\lambda_{s+1}\neq0$.
\end{enumerate}

\end{Lemma}

\begin{proof}
Fix $\sigma\in{\rm Sing}(X)$ such that $C(\sigma)$ is nontrivial and
denote $s={\rm Ind}(\sigma)$. By changing the Riemannian metric, we
can assume that $E^s(\sigma)\bot E^u(\sigma)$. Since $C(\sigma)$ is
nontrivial, there  exist $x\in C(\sigma)\cap
W^u(\sigma)\setminus{\{\sigma\}}$ and $y\in C(\sigma)\cap
W^s(\sigma)\setminus{\{\sigma\}}$. For any small $C^1$ neighborhood ${\cal U}$ of $X$, according to Lemma~\ref{Lem:Pseudo-Connecting} and its remark, there exists a neighborhood $V$ of $\sigma$, and $Y\in {\cal U}$ such that $Y=X$ on $V$ and $y=\phi_t^Y(x)$ for some $t>0$. By considering $\phi_N(y)$ and $\phi_{-N}(x)$ for $N>0$ large enough, we may assume that $x, y\in V$, which implies $\sigma_Y=\sigma$
exhibits a homoclinic orbit $\Gamma={\rm Orb}(z)$. Note that $X$ and $Y$ exhibit the same Lyapunov exponents at the singularity $\sigma_Y=\sigma$.

Choose two sequences
of regular points $x_n\rightarrow x$ and $y_n\rightarrow y$,
such that $\phi^Y_{t_n}(x_n)=y_n$. Connecting $x_n$ to $x$ and
$y_n$ to $y$, we derive a sequence of vector fields
$Y_n\rightarrow Y$ and $x_n\in {\rm Per}(Y_n)$ such that ${\rm
Orb}(x_n)$ converge to $\Gamma\cup\{\sigma\}$.

Considering the compact $Y$-invariant set $\Lambda=\Gamma\cup\{\sigma\}$, from Lemma \ref{Lem:Index-Estimation} we know that
$$
s-1\leq\lim_{n\rightarrow\infty}{\rm Ind}(x_n)\leq s
$$
which also implies that either $\beta(B^{s-1}(\Lambda))=\Lambda$ or $\beta(B^{s}(\Lambda))=\Lambda$. Assume the first case holds. Then the linear Poincar\'e flows $\psi^{Y_n}_t$ of all these periodic orbits admit the uniform dominated splitting
$$
\frac{\|\psi^{Y_n}_t|_{N^s(x)}\|}{m(\psi^{Y_n}_t|_{N^u(x)})}\le {\rm e}^{-2\eta t};
$$
for some constant $\eta>0$ and $\forall x\in{\rm Orb}(x_n)$, $\forall t\ge 1$. Since the constant $\eta$ is uniform for any $n$ and the extended linear Poincar\'e flow is a continuous linear flow on a continuous bundle, by taking limits in this framework, we get a dominated splitting
$$
\cN_{\Delta}={\cal E}\oplus {\cal F}
$$
over $\Delta$ with ${\rm dim}{\cal E}=s-1$, ${\rm dim}{\cal F}=d-s$. Here $\Delta\subset G^1$ is the set of limit points of $\{\langle Y_n(x)\rangle: x\in P_n\}$ and contained in $B^{s-1}(\Lambda)$. Since $P_n$ converges to the homoclinic loop, then we can choose $p_n\in{\rm Orb}(x_n)$ such that
$$
\lim_{n\to\infty}\langle X(p_n)\rangle\subset E^u(\sigma).
$$
This implies that $\Delta^u(\sigma)= \{L\in \Delta:
L\subset E^u(\sigma)\}$, which is a nonempty and compact invariant
set under $\Phi_t^Y=\Phi_t^X$ when restricted on $G^1(\sigma)=\{L\in G^1: \beta(L)=\sigma\}$. If we restrict the extended linear Poincar\'e
flow on $\cN_{\Delta^u(\sigma)}$, it will also admit the dominated
splitting with the same constant $2\eta$. Since we have assumed
$E^s(\sigma)\bot E^u(\sigma)$, we have
$$
E^s(\sigma)\subset\cN_{\Delta^u(\sigma)}
$$
and
$$
\psi^Y_t|_{E^s(\sigma)\cap\cN_{\Delta^{u}(\sigma)}}=\Phi^Y_t|_{E^s(\sigma)}.
$$
Since ${\rm dim}N^s=s-1$ and ${\rm dim}E^s(\sigma)=s$, $E^s(\sigma)$
admits a dominated splitting w.r.t. the tangent flow $\Phi^Y_t$ with
the same constant $\eta$, i.e.,
$$
E^s(\sigma)=E^{ss}(\sigma)\oplus E^c(\sigma)
$$
is a $\Phi^Y_t$-invariant splitting, where ${\rm dim}E^{ss}(\sigma)=s-1$ and ${\rm dim}E^c(\sigma)=1$. Moreover, it satisfies
$$
\frac{\parallel\Phi^Y_t|_{E^{ss}(\sigma)}\parallel}{m(\Phi^Y_t|_{E^c(\sigma)})}\leq {\rm e}^{-2\eta t}.
$$
This implies that the Lyapunov exponents of $\sigma_Y=\sigma$ satisfy
$\lambda_{s-1}\leq\lambda_s-2\eta$. Since $Y=X$ on a small neighborhood of $\sigma$, the same inequality holds for $X$.

If we assume $\beta(B^{s}(\Lambda))=\Lambda$, then the same analysis
shows that $\lambda_{s+1}\leq\lambda_{s+2}-2\eta$. This proves the
first item of this lemma.

For the rest three items, we need
\begin{Claim}
\begin{itemize}
\item If $\beta(B^{s}(\Lambda))=\Lambda$, then $\lambda_{s}+\lambda_{s+1}\leq-\eta$.
\item If $\beta(B^{s-1}(\Lambda))=\Lambda$, then $\lambda_{s}+\lambda_{s+1}\geq\eta$.
\end{itemize}
\end{Claim}

\begin{proof}[Proof of Claim:]
We just prove the first item, then for the second one we only need to consider $-Y$.
Recall the definition of $\beta(B^s(\Lambda))=\Lambda$, which means the homoclinic loop $\Lambda$ is the Hausdorff limit of periodic orbits Orb$(x_n)$ of $Y_n$ with index $s$. Applying Lemma \ref{Lem:Index-Estimation}, we know that there exists an $(s+1)$-dimensional subspace $E\subset T_{\sigma}M$, such that
$$
\frac{1}{k}\sum_{i=0}^{k-1}\log\|\wedge^2\Phi_1^Y|_{\Phi_i^Y(E)}\|\leq-\eta,\ k=1,2,\cdots
$$

On the other hand, $\beta(B^s(\Lambda))=\Lambda$ implies that $\lambda_{s+1}<\lambda_{s+2}$.
Denote by $E^{cs}$ the direct sum of the generalized eigenspaces associated to $\lambda_i, i=1,2,\cdots,s+1$, which is an $(s+1)$-dimensional $\Phi_t^Y$-invariant subspace of $T_{\sigma}M$. Then the dominated splitting on $T_{\sigma}M$ implies $E^{cs}$ must admit the estimation above, i.e.,
$$
\frac{1}{k}\sum_{i=0}^{k-1}\log\|\wedge^2\Phi_1^Y|_{\Phi_i^Y(E^{cs})}\|\leq-\eta,\ k=1,2,\cdots
$$

However, if we assume $\lambda_s+\lambda_{s+1}>-\eta$, we can pick a pair of eigenvectors $u$ and $v$ associated to $\lambda_s$ and $\lambda_{s+1}$ respectively. So we have the following equalities
$$
\|\Phi^Y_t(u)\|~=~ {\rm e}^{\lambda_st}\| u\|~, \qquad\forall t>0~,
$$
$$
\|\Phi^Y_t(v)\|~=~ {\rm e}^{\lambda_{s+1}t}\| v\|~, \qquad\forall t>0~.
$$
Since we have assumed $E^s(\sigma)\perp E^u(\sigma)$, which implies $u\perp v$, so we have
$$
\frac{1}{k}\sum_{i=0}^{k-1}\log\|\wedge^2\Phi_1^Y|_{\Phi_i^Y(E^{cs})}\|\geq\lambda_s+\lambda_{s+1}>-\eta,\ k=1,2,\cdots
$$
This is a contradiction. So we must have $\lambda_{s} + \lambda_{s+1} \leq -\eta$. This finishes the proof of the claim.
\end{proof}

Now we prove item 2 of this lemma. If $\lambda_{s-1}=\lambda_s$,
then by the analysis above, the homoclinic loop
$\Lambda=\Gamma\cup\{\sigma\}$ could only be accumulated by periodic
orbits of index $s$. This proves $\beta(B^{s}(\Lambda))=\Lambda$.
So we can apply the first item of the claim to show that $\lambda_s+\lambda_{s+1}\leq-\eta$.

Item 3 is just item 2 of $-X$.

Item 4 could be proved in the same way. In this case, we have two possibilities. Either $\beta(B^{s}(\Lambda))=\Lambda$ or $\beta(B^{s-1}(\Lambda))=\Lambda$. Corresponding to these two cases, the claim guarantee that we have either $\lambda_{s}+\lambda_{s+1}\leq-\eta$ or $\lambda_{s}+\lambda_{s+1}\geq\eta$. This finishes the proof of this lemma.
\end{proof}

\begin{Remark}
In the proof of this lemma, you can see that $|\lambda_s+\lambda_{s+1}|\geq\eta$. Moreover, either $\lambda_{s-1}$ and $\lambda_{s}$, or $\lambda_{s+1}$ and $\lambda_{s+2}$ should admit a uniform gap which is $2\eta$.
\end{Remark}

\begin{Corollary}\label{Cor:Saddle-Value}
For any $X\in{\cal X}^*(M^d)$ and any $\sigma \in {\rm Sing}(X)$, if
$C(\sigma)$ is nontrivial, then
$$
{\rm sv}(\sigma)\neq0.
$$
\end{Corollary}

\begin{Lemma}\label{Lem:Homoclinic-Loop}
Let $X\in{\cal X}^*(M^d)$ and $\sigma \in {\rm Sing}(X)$. Let
$\Gamma={\rm Orb}(x)$ be a homoclinic orbit associated to $\sigma$.
Assume there exists a sequence of vector fields $\{X_n\}$ converging
to $X$ in the $C^1$ topology and a sequence of  periodic orbits
$P_n$ of $X_n$ such that $P_n$ converges to $\Gamma\cup\{\sigma\}$
in the Hausdorff topology. Then we have
$$
\lim_{n\rightarrow\infty}{\rm Ind}(P_n)={\rm Ind}_p(\sigma),
$$
i.e., for $n$ large enough, ${\rm Ind}(P_n)={\rm Ind}_p(\sigma)$.
\end{Lemma}

\begin{proof}
We have proved that the saddle value of $\sigma$ is not equal to
zero. Without loss of generality we assume $\sv(\sigma)>0$,
otherwise we consider $-X$. Then the periodic index ${\rm
Ind}_p(\sigma)=s-1$, where $s={\rm Ind}(\sigma)$. Moreover, the
Lyapunov exponents of $\sigma$ satisfies
$\lambda_{s-1}<\lambda_{s}$, which determines a dominated
splitting of $T_{\sigma}M$:
$$
T_{\sigma}M=E^{ss}(\sigma)\oplus E^c(\sigma)\oplus E^u(\sigma).
$$
Here $E^c(\sigma)$ is the eigenspace associated to $\lambda_s$, and the saddle value $\sv(\sigma)>0$ insures that the invariant subspace $E^c(\sigma)\oplus E^u(\sigma)$ is sectional expanding.

Since Lemma \ref{Lem:Index-Estimation} has guaranteed that ${\rm
Ind}(P_n)\geq s-1$ for $n$ large enough, so we only need to show
that ${\rm Ind}(P_n)>s-1$ leads to a contradiction.
If ${\rm Ind}(P_n)>s-1$, also according to Lemma \ref{Lem:Index-Estimation},
$T_\sigma M$ contains a sectional contracting subspace $E$ of dimension $s+1$.

Then we have
$$
{\rm dim}(E\cap(E^c(\sigma)\oplus E^u(\sigma)))\geq{\rm dim}E+{\rm dim}(E^c(\sigma)\oplus E^u(\sigma))-d
$$
$$
\qquad\qquad\geq s+1+d-s+1-d=2.
$$
However, we notice that $E$ is sectional contracting and $E^c(\sigma)\oplus E^{u}(\sigma)$ is sectional expanding. This is absurd. So we have proved ${\rm Ind}(P_n)\leq s-1$ for $n$ large enough. This finishes the proof of the lemma.
\end{proof}

\begin{Remark}
This lemma asserts that when a periodic orbit is close enough to a homoclinic loop associated to some singularity, then its index has to be equal to the periodic index of the singularity. When we consider another kind of critical elements, periodic orbits, this also holds. Precisely, if the periodic orbit $Q_n$ tends to a homoclinic orbit $\Gamma=\{{\rm Orb}(x)\}$ associated to some periodic orbit $P$, then we must have ${\rm Ind}(Q_n)={\rm Ind}(P)$ for $n$ large enough. The reason here is that $\Gamma\cup P$ is a hyperbolic set since $\Gamma$ should be a transverse homoclinic loop (see \cite{GaW06}).
\end{Remark}

\begin{Lemma}\label{Lem:Homogeneous}
Let $X\in {\cal X}^*(M^d)$ be a  $C^1$ generic vector field and
$\sigma\in {\rm Sing}(X)$. Then for every critical element $c$ in
$C(\sigma)$,
$$
{\rm Ind}_p(c)={\rm Ind}_p(\sigma).
$$
\end{Lemma}

\begin{proof}
Here we take a $C^1$ generic $X\in{\cal X}^*(M^d)$ satisfying item 2 of Lemma \ref{Lem:Generic-Property}, i.e., if $p$ and $q$ are two different critical elements of $X$ with $C(p)=C(q)$, then there exists a $C^1$ neighborhood $\cal U$ of $X$ such that for any $Y\in{\cal U}$, one has $C(p_Y,Y)=C(q_Y,Y)$. Assume that there exists a critical element $c$ contained in $C(\sigma)$ such that
$$
{\rm Ind}_p(c)\neq{\rm Ind}_p(\sigma).
$$
Fix a $C^1$ neighborhood ${\cal U}\subset{\cal X}^*(M^d)$ as above and
all our perturbations will be contained in
$\cal U$. We will show that some perturbation $Z\in{\cal U}$ has a
periodic orbit with zero Lyapunov exponent, which is a
contradiction. First, we need the following sublemma.

\begin{Sublemma}\label{Lem:Heteroclinic-Cycle}
There exists $Y\in\cal U$ arbitrarily $C^1$ close to $X$ such that there is a heteroclinic cycle associated to $\sigma_Y$ and $c_Y$, i.e., there exist two regular points $x$ and $y$ such that
\begin{itemize}
\item ${\rm Orb}(x,Y)\subseteq W^s(\sigma_Y)\cap W^u(c_Y).$
\item ${\rm Orb}(y,Y)\subseteq W^u(\sigma_Y)\cap W^s(c_Y).$
\end{itemize}
\end{Sublemma}

\begin{proof}
If $c$ is not a singularity with index $s={\rm Ind}(\sigma)$, then
either
$$
{\rm dim}W^s(\sigma)+{\rm dim}W^u(c)\geq d+1,
$$
or
$$
{\rm dim}W^u(\sigma)+{\rm dim}W^s(c)\geq d+1.
$$
Without loss of generality we assume that the first case holds. Then
we can choose $x_s\in W^s(\sigma)\cap C(\sigma)$ and $x_u\in
W^u(c)\cap C(\sigma)$ and apply the connecting lemma for
chains (Lemma \ref{Lem:Pseudo-Connecting}) to create a
heteroclinic orbit
$$x\in W^s(\sigma_{X_1})\cap W^u(c_{X_1})$$
for some $X_1\in{\cal U}$. Moreover, since ${\rm
dim}W^s(\sigma)+{\rm dim}W^u(c)\geq d+1$, one can assume this
intersection is transverse after an arbitrary small $C^1$
perturbation when necessary. Since we still have
$C(\sigma_{X_1},X_1)=C(c_{X_1},X_1)$ which is nontrivial, we could
choose
$$
y_u\in W^u(\sigma_{X_1})\cap C(\sigma_{X_1})\qquad   {\rm and}\qquad    y_s\in W^s(c_{X_1})\cap C(\sigma_{X_1}).
$$
Moreover, we may assume that $X_1$ satisfies item 4 of Lemma
\ref{Lem:Generic-Property} so that you can apply the connecting lemma of
Wen-Xia (Lemma \ref{Lem:WX-Connecting}) to get some
$$y\in W^u(\sigma_{Y})\cap W^s(c_{Y})$$
for some $Y\in\cal U$ and $Y=X$ on $M\setminus {\rm Orb}(x)$ (see
the proof Theorem C in \cite{GaW03} for details). This finishes the
proof of the claim in the case that $c$ is not a singularity with
index $s={\rm Ind}(\sigma)$.

Now we assume that $c$ is a singularity with the same index of
$\sigma$. The difficulty here is that we could not achieve a
transverse heteroclinic orbit which will allow us to ``connect
twice". So we will need more assumptions on the vector field after
the first connecting.

First, we choose $x_s\in W^s(\sigma)\cap C(\sigma)$ and $x_u\in W^u(c)\cap C(\sigma)$ and applying the connecting lemma for chains to create a heteroclinic orbit
$$
\Gamma={\rm Orb}(x)\subseteq W^s(\sigma_{X_1})\cap W^u(c_{X_1}).
$$

Then we consider $\overline{W^u(\sigma_{X_1},X_1)}$, the closure of the unstable manifold of $\sigma_{X_1}$, which is lower semi-continuous with respect to $X_1$. Denote
$$
{\cal D}_{\Gamma}=\{S\in{\cal U}: S|_{\{\sigma_{X_1}\}\cup\Gamma\cup\{c_{X_1}\}}=X_1|_{\{\sigma_{X_1}\}\cup\Gamma\cup\{c_{X_1}\}}\}
$$
the set of all vector fields that coincide with $X_1$ on ${\{\sigma_{X_1}\}\cup\Gamma\cup\{c_{X_1}\}}$. Then ${\cal D}_{\Gamma}$ is a closed subset of ${\cal X}^1(M^d)$, which is also a $Baire$ set. This fact allows us to choose $X_2\in{\cal D}_{\Gamma}$ arbitrarily $C^1$ close to $X_1$, which is a continuous point of $\overline{W^u(\sigma_{X_2},X_2)}$ in ${\cal D}_{\Gamma}$.

\begin{Claim}
$$
c_{X_2}\in\overline{W^u(\sigma_{X_2},X_2)}.
$$
\end{Claim}

\begin{proof}[Proof of Claim:]
Otherwise, there exists an open neighborhood $V$ of $\overline{W^u(\sigma_{X_2},X_2)}\cap C(\sigma_{X_2})$, such that $c_{X_2}\in M^d\setminus\overline{V}$. We choose some $y\in W^u_{loc}(\sigma_{X_2},X_2)\cap C(\sigma_{X_2})\cap V$. Then $c_{X_2}$ is chain attainable from $y$, i.e., there exists a sequence of chain $\{y^n_i, t^n_i\}^{l_n}_{i=1},\ \forall t^n_i>T,\ n=1,2,3\cdots$ (for some $T>0$) which satisfy
$$
d(\phi^{X_2}_{t^n_i}(y^n_i),y^n_{i+1})<\frac{1}{n}, \qquad y^n_1=y \qquad  {\textrm and}\qquad d(\phi^{X_2}_{t^n_{l_n}}(y^n_{l_n}),c_{X_2})<\frac{1}{n},
$$
for all $1\leq i\leq l_{n-1}$ and $n>0$. Denote by $w_n$ the point
at which, for the first time, the chain $\{y^n_i,
t^n_i\}^{l_n}_{i=1}$ does not belong to $V$. Then $\{w_n\}$ will
converge to some point $w\in\partial V\cap C(\sigma_{X_2})$, which
does not belong to $\overline{W^u(\sigma_{X_2},X_2)}$. Moreover, we
assert that $w$ does not belong to $\Gamma$, otherwise the
chain before $w_n$ will accumulate to $c_{X_2}$ first, which
contradicts the fact that $w_n$ is the first point that escapes from
$V$. For the same reason, the Hausdorff limit of these chains
from $y$ to $w_n$ is far from $\Gamma$. We will use the connecting
lemma for chains here. One has
\begin{itemize}
\item There exists chains with arbitrarily small jumps from $y$ to $w$.
\item All these chains and their Hausdorff limit do not intersects $\overline{\Gamma}$.
\end{itemize}
By Lemma \ref{Lem:Pseudo-Connecting}, there is $X_3$ which is arbitrarily $C^1$ close to $X_2$, such that
\begin{itemize}
\item $w\in W^u(\sigma_{X_3},X_3)$.
\item The perturbation region does not intersect $\overline{\Gamma}$, which implies $X_3\in {\cal D}_{\Gamma}$.
\end{itemize}

This fact shows that we could enlarge
$\overline{W^u(\sigma_{X_2},X_2)}$ to $w$ by an arbitrarily small
$C^1$ perturbation in ${\cal D}_{\Gamma}$, which contradicts that
$X_2$ is a continuous point of $\overline{W^u(\sigma_{X_2},X_2)}$ in
${\cal D}_{\Gamma}$. This finishes the proof of the claim.
\end{proof}

Thus $c_{X_2}\in\overline{W^u(\sigma_{X_2},X_2)}$, which implies
that $c_{X_2}$ could be accumulated by some regular orbits contained
in $W^u(\sigma_{X_2},X_2)$. So there exists some point $z$ such that
$$
z\in\overline{W^u(\sigma_{X_2},X_2)}\cap W^s_{loc}(c_{X_2},x_2).
$$

One assumes that every $\varepsilon$-perturbation of $X_2$ is still in ${\cal U}$ for some $\varepsilon>0$. With the help of ${\rm C}^1$-connecting lemma of Wen-Xia (Lemma \ref{Lem:WX-Connecting}), for $\varepsilon>0$, there are $L>0$ and two neighborhoods $\widetilde{W_z}\subset W_z$ of $z$ such that if one denotes $W_{L,z}=\cup_{0\leq t\leq L}\phi^{X_2}_t(W_z)$, one has
\begin{itemize}
\item $W_{L,z}$ is disjoint from $\overline{\Gamma}$.
\item The positive orbit of some $y\in W^u_{loc}(\sigma_{X_2},X_2)$ intersects $\widetilde{W_z}$.
\end{itemize}
By Lemma \ref{Lem:WX-Connecting}, there is $Y$ $\varepsilon$-close to $X_2$ such that
\begin{itemize}
\item $Y$ has a heteroclinic orbit: ${\rm Orb}(y,Y)\subseteq W^u(\sigma_Y)\cap W^s(c_Y)$.
\item $\Gamma={\rm Orb}(x,Y)\subseteq W^s(\sigma_Y)\cap W^u(c_Y)$ is still a heteroclinic orbit.
\end{itemize}
This finishes the proof of the sublemma.
\end{proof}

Now we continue to prove Lemma \ref{Lem:Homogeneous}. For
simplicity, we will assume that $Y$ is $C^1$ linearizable around
$\sigma_Y$ and $c_Y$, and exhibits the heteroclinic cycle
$$
\Gamma_{0,0}=\Gamma_Y=\{\sigma_Y\}\cup\{c_Y\}\cup{\rm Orb}(x,Y)\cup{\rm Orb}(y,Y),
$$
where ${\rm Orb}(x,Y)\subseteq W^s(\sigma_Y)\cap W^u(c_Y)$ and ${\rm Orb}(y,Y)\subseteq W^u(\sigma_Y)\cap W^s(c_Y)$.

In two disjoint linearizable neighborhoods of $\sigma_Y$ and $c_Y$, choose two pairs of points $\{x_s, y_u\}$ and  $\{x_u, y_s\}$ such that
\begin{itemize}
\item $x_s\in W^s_{loc}(\sigma_Y)\cap{\rm Orb}(x)$ and $y_u\in W^u_{loc}(\sigma_Y)\cap{\rm Orb}(y)$,
\item $x_u\in W^u_{loc}(c_Y)\cap{\rm Orb}(x)$ and $y_s\in W^s_{loc}(c_Y)\cap{\rm Orb}(y)$.
\end{itemize}
Then we can choose two pairs of continuous
segments $\{x_{s,r}, y_{u,r}\}$, $0\le r\le 1$ and  $\{x_{u,t}, y_{s,t}\}$, $0\le t\le 1$ such that
\begin{itemize}
\item $\phi^Y_{t_r}(x_{s,r})=y_{u,r}$, $x_{s,0}=x_s$ and $y_{u,0}=y_u$;
\item $\phi^Y_{\tau_t}(y_{s,t})=x_{u,t}$, $x_{u,0}=x_u$ and $y_{s,0}=y_s$.
\end{itemize}

Connecting $x_s$ to $x_{s,r}$ and $y_u$ to $y_{u,r}$; $x_u$ to $x_{u,t}$ and
$y_s$ to $y_{s,t}$ continuously, we get a continuous family of star vector fields $\{Y_{r,t}:\ 0\leq r,t\le 1\}\subset{\cal U}\subset{\cal X}^*(M^d)$ with two parameters $r$ and $t$ such that
\begin{itemize}
\item $\lim_{r,t\to 0}Y_{r,t}=Y$.
\item $Y_{0,t}$ exhibits a homoclinic orbit associated to $\sigma_Y$, denoted by $\Gamma_{0,t}$ for $0\leq t\le 1$.
\item $Y_{r,0}$ exhibits a homoclinic orbit associated to $c_Y$, denoted by $\Gamma_{r,0}$ for $0\leq r\le 1$.
\item $Y_{r,t}$ exhibits a periodic orbit $\Gamma_{r,t}$ satisfying
$$
\lim_{r\to 0}\Gamma_{r,t}=\Gamma_{0,t} \qquad {\rm and} \qquad \lim_{t\to 0}\Gamma_{r,t}=\Gamma_{r,0}.
$$
\end{itemize}
We fix some $r_0>0$ and let $t\to 0$, for $t=t_0$ small enough, Lemma \ref{Lem:Homoclinic-Loop} insures that
$${\rm Ind}(\Gamma_{r_0,t_0})={\rm Ind}_p(c_Y).$$
Then letting $\Gamma_{r,t_0}\to \Gamma_{0,t_0}$ as $r\to 0$, and
applying Lemma \ref{Lem:Homoclinic-Loop} again, we know there is some
$r_1<r_0$ such that
$$
{\rm Ind}(\Gamma_{r_1,t_0})={\rm Ind}_p(\sigma_Y)\neq{\rm Ind}(\Gamma_{r_0,t_0}).
$$
Since the family of
vector fields $\{Y_{r,t_0}: r_1\le r\le r_0\}$ is continuous on the parameters $r$ in the $C^1$
topology, the Lyapunov exponents of $\Gamma_{r,t_0}$ is also
continuous on $r$. This implies that there must be some $r_2$ with
$r_1<r_2<r_0$, such that $\Gamma_{r_2,t_0}$ is a nonhyperbolic
periodic orbit, contradicting $Y_{r_2,t_0}\in{\cal U}\subset{\cal
X}^*(M^d)$. This finishes the proof of the lemma.
\end{proof}

\begin{Lemma}\label{Lem:Strong-Connection}
Let $X\in{\cal X}^*(M^d)$ and  $\sigma$ be a singularity of $X$ such
that $C(\sigma)$ is nontrivial. Then for every singularity $\rho$ in
$C(\sigma)$, we have
\begin{itemize}
    \item if $\sv(\rho)>0$, then $W^{ss}(\rho)\cap C(\sigma)=\{\rho\}$.
    \item if $\sv(\rho)<0$, then $W^{uu}(\rho)\cap C(\sigma)=\{\rho\}$.
\end{itemize}
\end{Lemma}

\begin{proof}
The proof of this lemma is the same with Lemma 4.3 of
\cite{LGW05}, and we just sketch it here. Assume $\sv(\rho)>0$ (if
$\sv(\rho)<0$ we consider $-X$). Then from Lemma
\ref{Lem:Expon-Sing} we know that there exists a dominated splitting
$$
T_{\rho}M=E^{ss}(\rho)\oplus E^c(\rho)\oplus E^u(\rho),
$$
which can be assumed to be mutually orthogonal.
Suppose on the contrary that $W^{ss}(\rho)\cap C(\sigma)\neq\{\rho\}$. Then applying the connecting lemma of chains, there exists some star vector field $Y$ arbitrarily $C^1$ close to $X$ exhibiting a strong homoclinic connection
$$
\Gamma\subset W^{ss}(\rho_Y,Y)\cap W^u(\rho_Y,Y).
$$
Moreover,we can assume $Y$ is linearizable around $\rho_Y$. Then using the perturbation around the singularities to generate periodic orbits accumulating the homoclinic loop, we get a sequence of vector fields $\{Y_n\}$ and $p_n\in{\rm Per}(Y_n)$ satisfying $p_n\to\rho$ and
$$
\langle Y_n(p_n)\rangle\hookrightarrow E^{ss}(\rho_Y)\oplus E^u(\rho_Y)\setminus (E^{ss}(\rho_Y)\cup E^u(\rho_Y)).
$$
Since we have ${\rm Ind}(p_n)={\rm Ind}_p(\rho_Y)={\rm
Ind}(\rho_Y)-1=s-1$, we can choose some nonzero $v$ such that
$L=\langle v\rangle\in B^{s-1}(C(\sigma_Y))$ and $v\in
E^{ss}(\rho_Y)\oplus E^u(\rho_Y)$. Let $v=v^{ss}+v^u$, where
$v^{ss}\in E^{ss}(\rho_Y)$ and $v^{u}\in E^{u}(\rho_Y)$. Without
loss of generality, we can assume that $|v^{ss}|=|v^u|$. Let
$w=v^{ss}-v^u$, then $v\perp w$. So, $(L,w)\in \cN_L$. Denote $(L_t,
w_t)=\psi^Y_t(L, w)$. Since $E^{ss}(\rho_Y)$ is contracting and
$E^u(\rho_Y)$ is expanding, we have
\begin{itemize}
\item $L_t\hookrightarrow E^u(\rho_Y)$ and $\langle w_t\rangle\hookrightarrow E^{ss}(\rho_Y)$, as $t\to +\infty$.
\item $L_t\hookrightarrow E^{ss}(\rho_Y)$ and $\langle w_t\rangle\hookrightarrow E^u(\rho_Y)$, as $t\to -\infty$.
\end{itemize}

There exists a dominated splitting $\cN_{B^{s-1}(C(\sigma_Y))\cap T_{\rho}M}=E\oplus F$ with index $s-1$, since $L$ is the limit of flow directions of periodic orbits. Now we consider two cases:

Case 1: $(L,w)\in E_L$.
 In this case, consider $t\to -\infty$. There exists $t_n\to -\infty$ such that $L_{t_n}\to L'\in E^{ss}(\rho_Y)$. According to the continuity of $E_L$, we know that $(L_{t_n},w_{t_n})\in E_{L_{t_n}}\to E_{L'}$. However we know that $\langle w_{t_n}\rangle\hookrightarrow E^u(\rho_Y)=F_{L'}$. This is a contradiction.

Case 2: $(L,w)\notin E_L$.
 In this case, consider $t\to +\infty$. There exists $t_n\to +\infty$ such that $L_{t_n}\to L'\in E^{u}(\rho_Y)$. Since $E\prec F$, we have $(L_{t_n},w_{t_n})\hookrightarrow F_{L'}$. However we know that $\langle w_{t_n}\rangle\hookrightarrow E^{ss}(\rho_Y)=E_{L'}$. This is also a contradiction.

This finishes the proof of Lemma \ref{Lem:Strong-Connection}.
\end{proof}

We end this section by summarizing these results to deduce Theorem
\ref{Thm:Anal-sing}.

\vskip 3mm \noindent{\bf Proof of Theorem \ref{Thm:Anal-sing}.~} We
take the dense $G_\delta$ subset ${\mathcal G_1}$ satisfying Lemma
\ref{Lem:Homogeneous}. Then Theorem \ref{Thm:Anal-sing} follows from
Corollary \ref{Cor:Saddle-Value}, Lemma \ref{Lem:Homogeneous} and
\ref{Lem:Strong-Connection} directly. This ends the proof of Theorem
\ref{Thm:Anal-sing}.\hfill $\Box$

\section{Singular hyperbolicity of singular chain recurrent classes}

In this section, we will give a proof of Theorem
\ref{Thm:Sing-Hyper}, which states that if all the singularities
contained in a singular chain recurrent class have the same index,
then this chain recurrent class must be
singular hyperbolic. During the proof, we will obtain a nice
description for the ergodic measures of star flows (Theorem
\ref{Thm:Erg-Measure}). The main techniques we will use are
Liao's Shadowing Lemma (Theorem \ref{Thm:Liao-Shadowing}) and his
estimation of the size of invariant manifolds (Theorem
\ref{Thm:Inv-Mfd-Size}).

First, we define quasi-hyperbolic arcs for the scaled linear
Poincar\'e flow (see Section 2 for definition).

\begin{Definition}\label{Def:Quasi-Hyp-arc}
Given $X\in{\cal X}^1(M^d)$ and $x\not\in{\rm Sing}(X)$, the orbit arc $\phi_{[0,T]}(x)$ is called $(\eta,T_0)^*$ quasi hyperbolic with respect to a direct sum splitting $\cN_x=E(x)\oplus F(x)$ and the scaled linear Poincar\'e flow $\psi^*_t$ if there exists $\eta>0$ and a partition
$$
0=t_0<t_1<\cdots<t_l=T, \textrm{ where  } t_{i+1}-t_i\in [T_0,2T_0]
$$
such that for $k=0,1,\cdots,l-1$, we have
$$\prod_{i=0}^{k-1}\parallel\psi^*_{t_{i+1}-t_i}|_{\psi_{t_i}(E(x))}\parallel\leq {\rm e}^{-\eta t_k},$$
$$\prod_{i=k}^{l-1}m(\psi^*_{t_{i+1}-t_i}|_{\psi_{t_i}(F(x))})\geq {\rm e}^{\eta(t_l-t_k)},$$
$$\frac{\parallel\psi^*_{t_{k+1}-t_k}|_{\psi_{t_k(E(x))}}\parallel}{m(\psi^*_{t_{k+1}-t_k}|_{\psi_{t_k}(F(x))})}\leq {\rm e}^{-\eta (t_{k+1}-t_k)}.$$
\end{Definition}

\begin{Remark}
This definition is similar to the usual quasi hyperbolic orbit arc
for linear Poincar\'e flow. The only difference is that we consider
the scaled linear Poincar\'e flow instead of the usual linear
Poincar\'e flow.
\end{Remark}

The proof of the next theorem could be found in \cite{Lia81} (see
\cite{GaY13} for more explanations).

\begin{Theorem}\label{Thm:Liao-Shadowing}(\cite{Lia81})
Given $X\in{\cal X}^1(M^d)$, a compact set $\Lambda\subset M^d\setminus{\rm Sing}(X)$, and $\eta>0, T_0>0$, for any $\varepsilon>0$ there exists $\delta>0$, such that for any $(\eta,T_0)^*$ quasi hyperbolic orbit arc $\phi_{[0,T]}(x)$ with respect to some direct sum splitting $\cN_x=E(x)\oplus F(x)$ and the scaled linear Poincar\'e flow $\psi^*_t$ which satisfies $x,\phi_T(x)\in\Lambda$ and $d(E(x),\psi_T(E(x)))\leq\delta$, there exists a point $p\in M^d$ and a $C^1$ strictly increasing function $\theta:[0,T]\to\RR$ such that
\begin{itemize}
\item $\theta(0)=0$ and $1-\varepsilon<\theta'(t)<1+\varepsilon$;
\item $p$ is a periodic point with $\phi_{\theta(T)}(p)=p$;
\item $d(\phi_t(x),\phi_{\theta(t)}(p))\leq\varepsilon|X(\phi_t(x))|,\  t\in[0,T].$
\end{itemize}
\end{Theorem}

\begin{Remark}
In this theorem, the compactness of $\Lambda$ guarantees the {\em
two ends} of the quasi hyperbolic string to be uniformly far from
the singularities. But we do not require the compact set $\Lambda$
to be invariant. Some part of the quasi hyperbolic string can be
very close to singularities. If the ends of the string are close to
singularity, the conclusion may not hold.
\end{Remark}

The second theorem of Liao we need is the significant estimation for
the size of invariant manifolds. Such kind of theorems are
well-known in the case of diffeomorphism and non-singular flow
(e.g., see \cite{PuS00}). If there is a singularity, however, it
would be very subtle and difficult to estimate the size of invariant
manifolds when the regular orbits approximate the singularity. As
before, we first introduce the definition of
$(\eta,T,E)^*$ contracting orbit arcs.

\begin{Definition}
Let $X\in{\cal X}^1(M^d)$, $\Lambda$  a compact invariant set of $X$, and $E\subset\cN_{\Lambda\setminus{\rm Sing}(X)}$ an invariant bundle of the linear Poincar\'e flow $\psi_t$. For $\eta>0$ and $T>0$, $x\in\Lambda\setminus{\rm Sing}(X)$ is called $(\eta,T,E)^*$ contracting if for any $n\in\NN$,
$$
\prod_{i=0}^{n-1}\parallel\psi^*_T|_{E(\phi_{iT}(x))}\parallel\leq {\rm e}^{-n\eta}.
$$
Similarly, $x\in\Lambda\setminus{\rm Sing}(X)$ is called $(\eta,T,E)^*$ expanding if it is $(\eta,T,E)^*$ contracting for $-X$.
\end{Definition}

\begin{Theorem}\label{Thm:Inv-Mfd-Size}(\cite{Lia89})
Let $X\in{\cal X}^1(M^d)$ and $\Lambda$ a compact invariant set of $X$. Given $\eta>0, T>0$, assume that $\cN_{\Lambda\setminus{\rm Sing}(X)}=E\oplus F$ is an $(\eta,T)$-dominated splitting with respect to the linear Poincar\'e flow. Then, for any $\varepsilon>0$, there is $\delta>0$ such that if $x$ is $(\eta,T,E)^*$ contracting, then there is a $C^1$ map $\kappa:E_x(\delta|X(x)|)\to\cN_x$ such that
\begin{itemize}
\item $d_{C^1}(\kappa,{\rm id})<\varepsilon$.
\item $\kappa(0)=0$.
\item $W^{cs}_{\delta|X(x)|}(x)\subset W^s({\rm Orb}(x))$, where $W^{cs}_{\delta|X(x)|}(x)={\rm exp}_x({\rm Image}(\kappa))$.
\end{itemize}
Here $E_x(r)=\{v\in E_x: |v|\le r\}.$
\end{Theorem}

\begin{Remark}
Compared with the cases of diffeomorphisms and non-singular flows,
we can see that this theorem is quite reasonable. In those two
cases, if we have a uniform contraction for the derivatives in the
future, we can achieve a uniform size of stable manifolds. But here,
because of the interference of singularities, we could only expect
the size of stable manifolds to be proportional to the flow speed. This
could also be thought as some kind uniform size of invariant
manifolds.
\end{Remark}

For the proof of Theorem \ref{Thm:Sing-Hyper}, we still need the
Ergodic Closing Lemma of Ma{\~n}{\'e}. We call a point $x\in M-{\rm
Sing}(X)$ is strongly closable for $X$, if for any $C^1$
neighborhood $\cal U$ of $X$, and any $\delta>0$, there exists
$Y\in\cal U$, $y\in M$, and $\tau>0$ such that the following
items are satisfied:
\begin{itemize}
\item $\phi^Y_{\tau}(y)=y.$
\item $d(\phi^X_t(x),\phi^Y_t(y))<\delta$, for any $0\leq t\leq\tau$.
\end{itemize}
The set of strongly closable points of $X$ will be denoted by
$\Sigma(X)$. The following flow version of the Ergodic Closing Lemma
can be found in $\cite{Wen96}$.

\begin{Theorem}\label{Thm:Erg-Clo-Lem}(\cite{Wen96})
For any $X\in{\cal X}^1(M)$, $\mu(Sing(X)\cup\Sigma(X))=1$ for every $T>0$ and every $\phi^X_T$-invariant Borel probability measure $\mu$.
\end{Theorem}

Now with the help of these theorems, we can give a  description for
ergodic measures of star flows. The next theorem is a detailed
version of Theorem E.

Given a $C^1$ vector field $X$, an ergodic measure $\mu$ of $X$ is called {\em  hyperbolic} if $\mu$ has at most one zero Lyapunov exponent, whose invariant subspace is spanned by $X$.

\begin{Theorem}\label{Thm:Erg-Measure}
Let $X\in{\cal X}^*(M^d)$. Then any ergodic
measure $\mu$ of $X$ is hyperbolic.
Moreover, if $\mu$ is not the atomic measure on any singularity,
then
$$
\supp(\mu)\cap H(P)\not=\emptyset,
$$
where $P$ is a periodic orbit with the index of $\mu$, i.e., the stable dimension of $P$ and $\mu$ coincide.
\end{Theorem}

\begin{proof}
Since $\mu$ is ergodic for the time-$t$ map $\phi_t$ except at most countable many $t$ (\cite{PS71}), we can choose $T$ large enough so that
\begin{itemize}
\item $\mu$ is ergodic for the time-$T$ map $\phi_T$,
\item  $\exists\eta>0$ such that the constants $T$ and $\eta$ satisfy the conclusion of Lemma \ref{Lem:Basic-Property}.
\end{itemize}

If $\mu$ supports on some critical element, then from the definition of star flows, it should be hyperbolic. So for the rest of the proof, we will assume that $\mu$ does not support on any critical element. We will first use the ergodic closing lemma to show $\mu$ is hyperbolic; then apply the argument of Katok and Liao's shadowing lemma (Theorem $\ref{Thm:Liao-Shadowing}$) to prove the existence of the accumulation of periodic orbits; and finally, the estimation of the size of stable and unstable manifolds (Theorem $\ref{Thm:Inv-Mfd-Size}$) will guarantee these periodic orbits are homoclinic related.

Applying Theorem $\ref{Thm:Erg-Clo-Lem}$, there exists some point
$x\in B(\mu)\cap\supp(\mu)\cap\Sigma(X)$ and $X_n\in{\cal
X}^1(M^d)$, $x_n\in M^d$, $\tau_n>0$ such that
\begin{itemize}
\item $\phi^{X_n}_{\tau_n}(x_n)=x_n$, where $\tau_n$ is the minimal period of $x_n$;
\item $d(\phi^{X}_t(x),\phi^{X_n}_t(x_n))<1/n$, for any $0<t<\tau_n$;
\item $\parallel X_n-X\parallel_{C^1}<1/n$.
\end{itemize}
Here $B(\mu)$ is the set of generic points of $\mu$. Recall that $x$ is a generic point of $\mu$ if for any continuous function $\xi: M^d\to\mathbb{R}$,
$$
\lim_{n\to+\infty}\frac 1n\sum_{i=0}^{n-1}\xi(\phi_{iT}(x))=\int\xi(y){\rm d}\mu(y).
$$
Since $\mu$ does not support on any critical element, we know that $\tau_n\rightarrow\infty$ as $n\rightarrow\infty$, and the ergodic measure $\mu_n$ supported on the periodic orbit of $x_n$ will converge to $\mu$ in the sense of weak topology. From Lemma $\ref{Lem:Basic-Property}$, we know that for any $x\in {\rm Orb}(x_n)$, $m\in\mathbb{N}$,
$$
\prod_{i=0}^{[{m\tau_n}/T]-1}\|\psi^{X_n}_T|_{N^s(\phi^{X_n}_{iT}(x))}\|\le {\rm e}^{-m\eta\tau_n},
$$
$$
\prod_{i=0}^{[{m\tau_n}/T]-1}m(\psi^{X_n}_T|_{N^u(\phi^{X_n}_{iT}(x))})\ge {\rm e}^{m\eta\tau_n}.
$$
These inequalities imply
$$
\int \log\parallel\psi^{X_n}_T\mid_{N^s(x)}\parallel d\mu_n(x)\leq -\eta,
$$
$$
\int \log m(\psi^{X_n}_{T}\mid_{N^u(x)}) d\mu_n(x)\ge \eta.
$$
We may assume that the index of ${\rm Orb}(x_n)$ is the same, then item 1 of Lemma \ref{Lem:Basic-Property} gives a dominated splitting on the limit: $N^s\oplus N^u$.
By considering the extended linear Poincar\'e flow $\psi_T^X(L,v)$, since $\psi$ is continuous in $T, X, L, v$ (see Lemma 3.1 in \cite{LGW05}),  we get that
$$
\int \log\parallel\psi^{X}_T\mid_{N^s(x)}\parallel d\mu(x)\leq -\eta,
$$
$$
\int \log m(\psi^{X}_{T}\mid_{N^u(x)}) d\mu(x)\ge \eta.
$$
This proves that $\mu$ is hyperbolic for $X$.

Since $\mu$ does not support on any critical element,
$$
\int \log\|\Phi_T|_{\langle X(x)\rangle}\| {\rm d}\mu(x)=0.
$$
We get that
$$
\int \log\parallel\psi^{*}_T\mid_{N^s(x)}\parallel d\mu(x)\leq -\eta,
$$
$$
\int \log m(\psi^{*}_{T}\mid_{N^u(x)}) d\mu(x)\ge \eta,
$$
equivalently,
$$
\int \log \|\psi^{*}_{-T}\mid_{N^u(x)}\| d\mu(x)\le -\eta.
$$
By Birkhoff Ergodic Theorem, we know that for $\mu-$almost every $z\in M$, we have
$$
\int \log\parallel\psi^*_T\mid_{N^s(x)}\parallel d\mu(x)=\lim_{k\rightarrow\infty}\frac{1}{k}\sum^{k-1}_{i=0}\log\parallel\psi^*_T\mid_{N^s(\phi^{X}_{iT}(z))}\parallel\leq -\eta,
$$
$$
\int \log \|\psi^*_{-T}\mid_{N^u(x)}\| d\mu(x)=\lim_{k\rightarrow\infty}\frac{1}{k}\sum^{k-1}_{i=0}\log \|\psi^*_{-T}\mid_{N^u(\phi^{X}_{-iT}(z))}\|\le -\eta.
$$

Following Katok's argument \cite{Kat80}, for every $K>0$, let $\Lambda_K$ be the set of points $x\in\supp(\mu)\cap B(\mu)$ such that for each $k>0$ one has
$$
\prod_{i=0}^{k-1}\|\psi^*_T|_{N^s(\phi^X_{iT}(x))}\|\le K{\rm e}^{-k\eta}, \qquad \qquad
\prod_{i=0}^{k-1}\|\psi^*_{-T}|_{N^u(\phi^X_{-iT}(x))}\|\le K{\rm e}^{-k\eta}.
$$
Then $\mu(\Lambda_K)\to 1$ as $K\to \infty$. So, for $K$ large enough, $\mu(\Lambda_K)>0$.
Since $\mu$ could not support on any critical element and is ergodic, we have $\mu(\Sing(X))=0$. So for some $\delta>0$,  $\Delta_K=\Lambda_K\setminus B(\Sing(X),\delta)$ has positive measure, where $B(\Sing(X),\delta)$ is the $\delta$-neighborhood of $\Sing(X)$ in $M$. Note that $\Delta_K$ is a closed set. According to Poincar\'e recurrence theorem, this implies that for every $z\in \supp{\mu|_{\Delta_K}}$, one can find orbit arcs $\phi_{[0, m_nT]}(x_n)$ such that  $x_n, \phi_{m_nT}(x_n)$ belong to $\Delta_K$, the distances $d (x_n, z), d(z, \phi_{m_nT}(x_n))$ are arbitrarily small and the non-invariant atomic measure
$$
\mu_n=\frac 1{m_n} \sum_{i=0}^{m_n-1}\delta_{\phi_{iT}(x_n)}
$$ is arbitrarily close to $\mu$. In particular, for each $0\leq k\leq m_n$ we have
$$
\prod_{i=0}^{k-1}\|\psi^*_T|_{N^s(\phi^X_{iT}(x_n))}\|\le K{\rm e}^{-k\eta}, \qquad and \qquad
\prod_{i=0}^{k-1}\|\psi^*_{-T}|_{N^u(\phi^X_{(n-i)T}(x_n))}\|\le K{\rm e}^{-k\eta}.
$$
Since here the end points of the quasi-hyperbolic orbit arc are uniformly $\delta-$away from the singularities of $X$, we can apply the shadowing lemma of Liao (Theorem $\ref{Thm:Liao-Shadowing}$): there exists a sequence of periodic points $p_n$ converge to $z$, such that the atomic measure supported on ${\rm Orb}(p_n, X)$ converge to $\mu$. Moreover, the property of shadowing original $(\eta,T)^*$ quasi hyperbolic orbit arcs guarantees that some $q_n\in{\rm Orb}(p_n)$ is $(\eta/2,T,N^s)^*$ contracting and $(\eta/2,T,N^u)^*$ expanding: for $n$ large enough (to eliminate the constant K) and every $k\in\NN$,
$$
\prod_{i=0}^{k-1}\|\psi^*_{T}|_{N^s(\phi^X_{iT}(q_n))}\|\le {\rm e}^{-k\eta/2},
$$
$$
\prod_{i=0}^{k-1}\|\psi^*_{-T}|_{N^u(\phi^X_{-iT}(q_n))}\|\le {\rm e}^{-k\eta/2}.
$$
Then Theorem $\ref{Thm:Inv-Mfd-Size}$ shows that $q_n$ will have a uniform size of local stable and unstable manifolds, which guarantees that for $n$ large enough, periodic orbits ${\rm Orb}(p_n)={\rm Orb}(q_n)$ are mutually homoclinic related, and hence $z\in H(p_n)$.
This finishes the proof of the theorem.
\end{proof}
\begin{Remark}
\begin{enumerate}
\item Theorem E is the first conclusion of this theorem.
\item From the proof, we can see that points in $\Lambda_K$ close to $z$ also belong to $H(P)$. So, $\mu(H(P))>0$. Since $\mu$ is ergodic, $\mu(H(P))=1$, i.e., $\mu$ is supported on $H(P)$. Especially,
$$
\supp(\mu)\subseteq \overline{{\rm Per}(X)}.
$$
\item According to Theorem \ref{Thm:Erg-Measure}, if $\mu$ is a nontrivial ergodic measure of a star vector field, then the measurable entropy of $\mu$ is positive.
\end{enumerate}
\end{Remark}
Applying the description of invariant measures of star flows, we
prove the following  {\bf homogeneous property} for generic star
flows.

\begin{Theorem}\label{Thm:Loc-Homoge}
For a $C^1$ generic star vector field $X$ and any chain recurrent
class $C$ of $X$, there exists a neighborhood $U$ of $C$ such that
all the critical elements contained in $U$ have the same periodic
index with the critical elements contained in $C$.
\end{Theorem}

\begin{proof}
For the case where $C$ does not contain any singularities, we refer to $\cite{GaW06}$ which showed that the homogeneous property holds for the nonsingular chain recurrent class of any star vector fields. So we will focus on the case where $C=C(\sigma)$ is nontrivial and the vector field $X$ satisfies the generic properties which will guarantee the conclusion of Lemma $\ref{Lem:Homogeneous}$.

Now we assume that there exists a sequence of periodic orbits $\{P_n\}$ whose Hausdorff limit is contained in $C(\sigma)$, and
$$
{\rm Ind}_p(P_n)={\rm Ind}(P_n)=k\neq{\rm Ind}_p(\sigma).
$$
Without loss of generality, we may assume that $k>{\rm Ind}_p(\sigma)$.

The invariant probability measure $\mu_n$ supported on $P_n$ will converge to an invariant measure $\widetilde{\mu}$ whose support is contained in $C(\sigma)$. Denote by
$$
\xi(x)=\inf_{E\subset T_x M, \dim E=k+1}~\sup_{L\subset E, \dim L=2}\log |\det \left(\Phi_T|_L\right)|.
$$
It is easily seen that $\xi: M\to \mathbb{R}$ is continuous. Since
$$
\int \xi(x){\rm d}\mu_n\le-\eta,
$$
we have
$$
\int \xi(x){\rm d}\widetilde{\mu}\le-\eta.
$$
Then, the Ergodic Decomposition Theorem allows us to find an ergodic invariant measure $\mu$ supported on $C(\sigma)$ which also satisfies the the above estimation
$$
\int \xi(x){\rm d}\mu\le-\eta.
$$
Obviously, $\mu$ could not support on any singularity in $C(\sigma)$.
Theorem $\ref{Thm:Erg-Measure}$ tells us that $\mu$ is hyperbolic with index $\ge k$ and
$\supp(\mu)\cap H(q)\not=\emptyset$ for some periodic point $q$ with index $\ge k$. By the definition of chain recurrent class and homoclinic class, we know that $q\in C(\sigma)$. However, this is impossible because ${\rm Ind}_p(q)\ge k>{\rm Ind}_p(\sigma)$ which contradicts to the conclusion of Lemma $\ref{Lem:Homogeneous}$. This finishes the proof of the theorem.
\end{proof}

Now we can finish the proof of Theorem $\ref{Thm:Sing-Hyper}$ with the help of the description of ergodic measures and the homogeneous property of star vector fields.

\vskip 3mm
\noindent{\bf Proof of Theorem \ref{Thm:Sing-Hyper}.~} We take the dense $G_\delta$ subset ${\mathcal G_2}\subseteq{\mathcal G_1}\subseteq{\cal X}^*(M^d)$ whose elements also satisfy the generic properties stated in Theorem $\ref{Thm:Loc-Homoge}$ and the fourth item of Lemma $\ref{Lem:Generic-Property}$. For any $X\in{\mathcal G_2}$ and a nontrivial chain recurrent class $C(\sigma)$ where $\sigma\in{\rm Sing}(X)$, from Lemma $\ref{Lem:Generic-Property}$ we know that there exists a sequence of periodic orbits $\{Q_n\}$ converge to $C(\sigma)$ in the Hausdorff topology. Without loss of generality, we may assume that ${\rm sv}(\sigma)>0$. By Theorem \ref{Thm:Anal-sing} and the conclusion of Lemma \ref{Lem:Homogeneous}, for any $\rho\in{\Sing}(X)\cap C(\sigma)$, ${\rm sv}(\rho)>0$. Moreover, the homogeneous property and  $W^{ss}(\rho)\cap C(\sigma)=\{\rho\}$ (from Theorem \ref{Thm:Anal-sing}) for any $\rho\in C(\sigma)\cap{\rm Sing}(X)$ guarantees that
$$
{\rm Ind}(Q_n)={\rm Ind}_p(\rho)={\rm dim}E^{ss}(\rho), \qquad \forall \rho\in C(\sigma)\cap{\rm Sing}(X).
$$
This implies $\beta(B^k(C(\sigma)))=C(\sigma)$ (where $k={\rm dim}E^{ss}(\rho)$) and it has a continuous splitting of the extended tangent flow over the compactification of $C(\sigma)$:
$$
\beta^*(T_{C(\sigma)}M^d)\mid_{B^k(C(\sigma))}=N^s\oplus\cP\oplus N^u.
$$
Recall that $\cP$ is the limit of flow line, which is $\Phi_t$-invariant. $N^{s/u}$ are contained in the normal bundle, which is invariant by the extended linear Poincar\'e flow $\psi_t$, and $E^{cs/cu}= N^{s/u}\oplus\cP$ is $\Phi_t$-invariant. Changing the metric if necessary, we can assume that $E^{ss}(\rho)\perp E^{cu}(\rho)$ for any singularity $\rho$. Since $W^{ss}(\rho)\cap C(\sigma)=\{\rho\}$, we know that $\cP\mid_{B^k(\{\rho\})}\subseteq B^k(\{\rho\})\times E^{cu}(\rho)$. Consequently, the domination of the extended linear Poincar\'e flow $N^s\prec N^u$ ensures that
$$
N^s\mid_{B^k(\{\rho\})}=B^k(\{\rho\})\times E^{ss}(\rho), \qquad \forall \rho\in C(\sigma)\cap{\rm Sing}(X).
$$

\begin{Claim}
There exists a mixed dominated splitting $(N^s,\psi_t)\prec(\cP,\Phi_t)$ on $B^k(C(\sigma))$, i.e., there exists $T>0$ such that
$$
\frac{\|\psi_T|_{N^s}\|}{m(\Phi_T|_{\cP})}\le \frac 12.
$$
\end{Claim}

\begin{proof}[Proof of Claim:]
The claim is equivalent to say that the scaled linear Poincar\'e flow $\psi^*_t$ restricted on $N^s$ is uniformly contracting. If it is not uniformly contracting, then there exists an ergodic invariant measure $\mu$ whose support is contained in $C(\sigma)$ such that
$$
\int \log\parallel\psi^*_T\mid_{N^s(x)}\parallel {\rm d}{\mu}(x)\geq0.
$$
It is easy to see that the push-forward measure $\beta_*(\mu)$ on $M$ can not to be the atomic measure at singularity since $\cP\mid_{B^k(\{\rho\})}\subseteq B^k(\{\rho\})\times E^{cu}(\rho)$ and $N^s\mid_{B^k(\{\rho\})}=B^k(\{\rho\})\times E^{ss}(\rho)$. So, the above inequality is also satisfied for the measure $\beta_*(\mu)$ on $M$. Moreover, the inequality implies that the dimension of invariant subspace associated to negative Lyapunov exponents of (the hyperbolic measure) $\beta_*(\mu)$ is less than $k$. Theorem \ref{Thm:Erg-Measure} tells us that $\supp(\beta_*(\mu))\cap H(P)\not=\emptyset$ for some periodic orbit $P$ with index less than $k$. This contradicts to the homogeneous property stated in Lemma $\ref{Lem:Homogeneous}$. This finishes the proof of the claim.
\end{proof}

Since $\cP\mid_{B^k(\{\rho\})}\subseteq B^k(\{\rho\})\times E^{cu}(\rho)$, a similar proof as the above claim shows that $N^s$ is uniformly contracting with respect to $\psi_t$.

The rest part of the proof is to show that $\Phi_t$ admits a partially hyperbolic splitting over $T_{C(\sigma)}M$. This is almost exactly the same as the proof Theorem A in $\cite{LGW05}$, and we just sketch the proof for the convenience of reader. By Lemma $\ref{Lem:Basic-Property}$ and the claim we have
$$
(N^s,\psi_t)\prec(N^u,\psi_t)\qquad {\rm and} \qquad (N^s,\psi_t)\prec(\cP,\Phi_t).
$$
According to Lemma 5.5 of $\cite{LGW05}$ (see also $\cite{BDP03}$, Lemma 4.4) the above dominations imply that  we have the mixing dominated splitting $(N^s, \psi_t)\prec_{T_0}(E^{cu},\Phi_t)$ for some $T_0>0$. So the linear bundle map
$$
\Phi_{T_0} : \beta^*(TM)\mid_{B^k(C(\sigma))}\rightarrow\beta^*(TM)\mid_{B^k(C(\sigma))},
$$
can be expressed as
$$
\Phi_{T_0}=\left(
            \begin{array}{cc}
              A & 0 \\
              C & D \\
            \end{array}
          \right): N^s\oplus E^{cu}\rightarrow N^s\oplus E^{cu},
$$
where $A=\psi_{T_0}\mid_{N^s}$, $D=\Phi_{T_0}\mid_{E^{cu}}$. Moreover the mixed domination $(N^s, \psi_t)\prec_{T_0}(E^{cu},\Phi_t)$ implies that
$$
\frac{\|A\|}{m(D)}\le \frac 12.
$$
Then the calculation in $\cite[Lemma \ 5.6]{LGW05}$ tells us there exists a $\Phi_{T_0}$-invariant subbundle, denoted by $E^{ss}$. This give a continuous $\Phi_{T_0}$-invariant splitting
$$
\beta^*(TM)\mid_{B^k(C(\sigma))}=E^{ss}\oplus E^{cu}.
$$
The compactness of $B^k(C(\sigma))$ and continuity of this invariant splitting guarantees there exists some constant $C>0$, such that
$$
\parallel\Phi_{T_0}\mid_{E^{ss}}\parallel\leq C\parallel\psi_{T_0}\mid_{N^s}\parallel.
$$
Finally, for any $t>0$, $E^{ss}\oplus E^{cu}$ is a $\Phi_t$-invariant dominated splitting by the uniqueness of dominated splitting. And this splitting induces a $\Phi_t$-invariant dominated splitting $E^{ss}\oplus E^{cu}$ on $T_{C(\sigma)}M$, and $E^{ss}$ is uniformly contracting with respect to $\Phi_t$ since $N^s$ is uniformly contracting with respect to $\psi_t$.

Now we have proved that
$$
T_{C(\sigma)}M=E^{ss}\oplus E^{cu}
$$
is a $\Phi_t$-invariant partially hyperbolic splitting. For the
singular hyperbolicity, we only need to show that
$\Phi_t\mid_{E^{cu}}$ is sectional expanding. This is exactly the
same as the proof of the claim. If it is not, then we can find an
ergodic measure on $C(\sigma)$ such that its dimension of stable
bundle is larger than $k={\rm Ind}_p(\sigma)$. The fact that the
saddle values of all the singularities contained in $C(\sigma)$ are
larger than 0 excludes the possibility that this measure is an
atomic measure at any singularity. Then Theorem
$\ref{Thm:Erg-Measure}$ allows us to find a periodic orbit contained
in $C(\sigma)$ whose index is larger than $k$. This contradicts the
homogeneous property of $X$, and finishes the proof of this theorem.
\hfill $\Box$

\bigskip
\noindent
Yi Shi $<$polarbearsy@gmail.com$>$: School of Mathematical Sciences, Peking University, Beijing 100871, China {\em $\&$}
\\
Institut de Math\'ematiques de Bourgogne, Universit\'e de Bourgogne, Dijon 21000, France\\

\noindent
Shaobo Gan $<$gansb@pku.edu.cn$>$: School of Mathematical Sciences, Peking University, Beijing 100871, China\\

\noindent
Lan Wen $<$lwen@math.pku.edu.cn$>$: School of Mathematical Sciences, Peking University, Beijing 100871, China\\

\end{document}